\documentclass[11pt,a4paper]{article}
\usepackage{authblk}
\usepackage[margin=2cm]{geometry}
\usepackage{t1enc}
\usepackage[utf8]{inputenc}
\usepackage{amsthm,amsmath,amssymb}
\usepackage{graphicx}
\usepackage{enumerate}
\usepackage{hyperref}
\usepackage{bm}
\usepackage{comment}
\usepackage{amsfonts}
\usepackage{graphicx,caption}
\usepackage{bm}
\usepackage{amsmath, amsthm, amssymb}
\usepackage{graphicx}
\usepackage{hyperref}
 \usepackage{relsize}
 \usepackage{xcolor}
 \usepackage{authblk}
\usepackage{algpseudocode}
\usepackage{mathtools}

\usepackage{bbm}

\theoremstyle{plain}
\usepackage{amsthm}
\makeatletter
\newcommand{\newreptheorem}[2]{\newtheorem*{rep@#1}{\rep@title}\newenvironment{rep#1}[1]{\def\rep@title{#2 \ref*{##1}}\begin{rep@#1}}{\end{rep@#1}}}
\makeatother

\DeclarePairedDelimiter\abs{\lvert}{\rvert}%
\DeclarePairedDelimiter\norm{\lVert}{\rVert}%
\makeatletter
\let\oldabs\abs
\def\abs{\@ifstar{\oldabs}{\oldabs*}}
\let\oldnorm\norm
\def\norm{\@ifstar{\oldnorm}{\oldnorm*}}
\makeatother

\theoremstyle{plain}

\newtheorem{theorem}{Theorem}[section]
\newtheorem{lemma}[theorem]{Lemma}
\newtheorem{corollary}[theorem]{Corollary}

\theoremstyle{remark}
\newtheorem{definition}[theorem]{Definition}

\newtheorem{remark}[theorem]{Remark}

\newcommand{\newtau}{h}

\newcommand{\Z}{\mathbb{Z}}
\newcommand{\R}{\mathbb{R}}
\newcommand{\N}{\mathbb{N}}
\newcommand{\Q}{\mathbb{Q}}
\newcommand{\eps}{\varepsilon}

 \newcommand{\1}{\mathbbm{1}}
\newcommand{\Prob}[1]{\mathbb{P}\left(#1\right)}

\newcommand{\Probbig}[1]{\mathbb{P}\big(#1\big)}
\newcommand{\ProbBig}[1]{\mathbb{P}\Big(#1\Big)}
\newcommand{\Exp}{\mathbb{E}}
\newcommand{\ud}{\mathrm{d}}

\DeclareMathOperator*{\argmin}{argmin}

\DeclareMathOperator*{\range}{Range}
\DeclareMathOperator*{\supp}{supp}

\newcommand{\ov}[1]{\newbar{#1}}

\newcommand\restr[2]{{% we make the whole thing an ordinary symbol
  \left.\kern-\nulldelimiterspace % automatically resize the bar with \right
  #1 % the function
  \vphantom{\big|} % pretend it's a little taller at normal size
  \right|_{#2} % this is the delimiter
  }}

\newcommand{\action}{A}
\newcommand{\maxPath}{\newbar{\gamma}}

\newcommand{\discr}[1]{d(#1)}
\newcommand{\spacevar}[1]{F(#1)}

\newcommand{\round}[1]{\left[ #1 \right]}
\newcommand{\paths}{\Gamma}
\newcommand{\rectAction}{\ov{A}}
\newcommand{\shape}{\Lambda}

\newcommand{\rectShape}{\Psi}
\newcommand{\cone}{\mathcal{K}}

\newcommand{\Ng}{\mathfrak{N}}
\newcommand{\Fc}{\mathcal{F}}
\newcommand{\Bc}{\mathcal{B}}
\newcommand{\Nc}{\mathcal{N}}
\newcommand{\Uc}{\mathcal{U}}
\newcommand{\Ic}{\mathcal{I}}

\newcommand{\Pp}{\mathbb{P}}
\newcommand{\E}{\mathbb{E}}
\newcommand{\cov}{\mathsf{cov}}

\newcommand{\probBi}{r}
\newcommand{\epf}{{{\hfill $\Box$ \smallskip}}}
\newcommand{\tj}{j}

\usepackage{xcolor}

\newif\ifshowcomments
\showcommentstrue %comment this out if want to hide comments

\usepackage{cancel}

\newcommand{\newbar}[1]{{#1}^*}

\author[1]{Yuri Bakhtin}
\author[2,3]{Konstantin Khanin}
\author[2,4]{Andr\'as M\'esz\'aros}
\author[2]{Jeremy Voltz}

\affil[1]{Courant Institute of Mathematical Sciences, New York University, USA}
\affil[2]{Department of Mathematics, University of Toronto, Canada}
\affil[3]{Beijing Institute of Mathematical Sciences and Applications (BIMSA), Beijing, China}
\affil[4]{HUN-REN Alfr\'ed R\'enyi Institute of Mathematics, Budapest, Hungary}

\title{Last Passage Percolation in a Product-Type Random Environment} 
 
\date{}
 
\begin{document}

\maketitle

\begin{abstract}
    We consider a last passage percolation model in dimension $1+1$ with potential given by the product of a spatial i.i.d.\ potential with symmetric bounded distribution and an independent i.i.d.\ in time sequence of signs. We assume that the density of the spatial potential near the edge of its support behaves as a power, with exponent $\kappa>-1$. 
    We investigate the linear growth rate of the actions of optimal point-to-point lazy random walk paths as a function of the path slope and describe the structure of the resulting shape function. It has a corner at $0$ and, although its restriction to positive slopes cannot be linear, we prove that it has a flat edge near~$0$ if $\kappa>0$.
   
    For optimal point-to-line paths, we study their actions and  locations of favorable edges that the paths tends to reach and stay at.
    Under an additional assumption on the time it takes for the optimal path to reach the favorable location, we prove that appropriately normalized actions converge to a limiting distribution that can be viewed as a counterpart of the Tracy--Widom law. Since the scaling exponent and the limiting distribution depend only on the parameter~$\kappa$, our results provide a description of a new universality class.

    \end{abstract}

\section{Introduction}
The term {\it directed polymers} refers to a class of models describing elastic chains interacting with their random environment. Such models have been intensively studied in the last 40 years.

A standard model in the discrete setting is based on nearest-neighbour random walk paths $\gamma_n(i), \, 0\leq i \leq n$, of length $n\in\N$ on the integer lattice 
$\Z^d$. In this paper, we will work with lazy walks, that is, $\Z^d$-valued paths satisfying  $\|\gamma(i+1)-\gamma(i)\|\leq 1$, where $\|\cdot\|$ is the Euclidean norm. It is natural to treat the argument $i$ as time. Although we give the basic definitions for a general $d$, later we will restrict our attention to the case $d=1$.

The random environment is given by a random potential $\Phi=(\Phi(x,i), \, x\in \Z^d, i\in \Z)$ which is time-dependent in general. 

The random 
energy or action accumulated by a path $\gamma_n$ from the environment is given by
\begin{equation}
\label{eq:action-general}
A_n(\Phi;\gamma_n)=\sum_{i=1}^n\Phi(\gamma_n(i),i).
\end{equation}

Given the inverse temperature $\beta>0$ and 
a realization of the random potential $\Phi$, one can define the polymer measure which is the probability distribution on lazy walk 
paths started at the origin, i.e., satisfying $\gamma_n(0)=0$,
by
\begin{equation*}
P_n(\Phi; \gamma_n)=\frac{1}{Z_n(\Phi;\beta)}\exp{\{-\beta A_n(\Phi;\gamma_n)\}},    
\end{equation*}
where  the random partition function 
$Z_n(\Phi;\beta)=\sum_{\gamma_n}\exp{\{-\beta A_n(\Phi;\gamma_n)\}}$ is a normalizing factor.

We are interested in the asymptotic
behavior as $n \to \infty$ of the polymer measure and of the 
partition function  $Z_n(\Phi;\beta)$ for typical realizations of the environment $\Phi$. In the zero temperature case where $\beta= +\infty $, the polymer measure is concentrated on the ground states, i.e.,
action-minimizing paths
$\newbar{\gamma}_n(\Phi)= \argmin_{\gamma_n} {A_n(\Phi, \gamma_n)}$.  These minimizers and their actions can be viewed as a last passage percolation (LPP) model.\footnote{A note on terminology: One could argue that since we want to minimize the action, it would be more appropriate to refer to our model as a first passage percolation model. However, since time is directed in our case, we decided to go with the name last passage percolation which is used more often for directed models.} 
One often makes assumptions on the distribution of $\Phi$ ensuring that
for almost every realization of $\Phi$ and for all $n\in\N$, there is a unique path minimizing $A_n(\Phi,\cdot)$ among those starting at $0$. Thus, we are interested in asymptotic geometrical properties of minimizers $\newbar{\gamma}_n=\newbar{\gamma}_n(\Phi)$ and their actions $\newbar{A}_n=A_n(\Phi; \newbar{\gamma}_n)$.

The setting where  the random variables $\Phi(x,i)$ are i.i.d.\ is the most studied in
the mathematical and physical literature. The case $d=1$ corresponding to the KPZ phenomenon is of special interest. It is predicted 
that under very mild conditions on the distribution of $\Phi(x,i)$, the asymptotic statistical properties of optimal  paths and their actions are universal.
Moreover, the same universal behavior is expected for positive and zero temperature. For zero temperature this, in particular, means that there exist nonrandom constants $A, \newtau_1$ such that as $n \to \infty $
$$ \frac{\newbar{A}_n -An}{\newtau_1 n^{1/3}} \stackrel{d}{\longrightarrow}TW_{GOE},$$
where $TW_{GOE}$ is the Tracy--Widom distribution for the Gaussian Orthogonal Ensemble. Furthermore, for some positive constant $\newtau_2$, the distribution of 
$$\frac{\newbar{\gamma}_n(\Phi)}{\newtau_2n^{2/3}}$$  also converges to a universal law. The area of KPZ was extremely active in the last 20 years (\cite{PR_SP, AQC, BOR, SEP, DOT, PD, MQR, DV}). The literature on KPZ phenomenon is huge. Here we just cited 
a few papers where the reader can find other references.
We also want to point out that most of the results are related to
exact formulas for particular models, that is for a particular
distribution of the disorder potential $\Phi$. At the same time,
the problem of universality remains largely open. 

In this paper, we consider $d=1$ and a random potential $\Phi$ of a different type, given by
\[\Phi(x,i)=F(x)B(i),\]
where
 $F=(F(x))_{x \in \Z}$ and $B=(B(i))_{i\in\Z}$ 
 are two mutually independent families of i.i.d.\ random variables.
The action~\eqref{eq:action-general} then can be rewritten as
\begin{equation}
\label{eq:action-product}
A_n(B,F;\gamma_n)=\sum_{i=1}^n F(\gamma_n(i))B(i).
\end{equation}

Additionally, we assume that random variables $B(i)$ take values $\pm1$ with probabilities~$1/2$ and random variables
$F(x)$ have a density $\varrho(x)$ which is a positive, even and continuous function on an open interval $(-c,c)$ which vanishes outside of this interval:
\begin{align}\nonumber
& \varrho(x)=\varrho(-x), \, \varrho(x)>0 \, &\text{ for } x \in (-c,c), \\& \varrho(x)=0 &\text{ for } x \notin (-c,c),\label{p(x)} \\ 
& \lim_{x\to c} \frac{\varrho(x)}{|c-x|^\kappa}=q, &\text{ where } q>0, \quad -1< \kappa.\nonumber
\end{align}

Note that since $\varrho$ is even, we also have $\lim_{x\to -c} \frac{\varrho(x)}{|c+x|^\kappa}=q$.

The exponent $\kappa$ determines the asymptotic behavior of $\varrho(x)$ near the boundary of its support.  The condition $\kappa>-1$ is necessary, because the density needs to be integrable. For $\kappa>0$, the density $\varrho(x)$ is a continuous function on $\R^1$.

There are several motivations for considering such a random environment.
From the physics point of view the model corresponds to a spatially disordered media which interacts with external time-dependent forces.
The product structure corresponds to the situation where the spatial correlation length of the time-dependent external forces is much larger than the correlation length of the spatially disordered media.

We are also motivated by the comparison of our model with the models from the KPZ universality class. Our model is simpler and very flexible since it allows for a broad class of densities $\varrho$. We will see that the large scale properties of our model depend only on the parameter $\kappa$, see~\eqref{p(x)}, thus defining a new universality class. 
Some of these large scale properties are drastically different from the KPZ behavior.  
This concerns, for example, the regularity properties of the shape functions characterizing the dependence of the linear rate of growth of action of optimal point-to-point paths on the path slope. For models in the KPZ universality class, the shape functions conjecturally have no corners or flat edges.   Differentiability of shape functions for a class of LPP-type models in continuous space was recently proved in~\cite{bakhtin-dow-2023-1}, \cite{bakhtin-dow-2023-2}.
 In contrast,  the shape function for our model is not smooth, with a corner at $0$ (this feature is more typical for First Passage Percolation models with static random environment~\cite{haggstrom1995asymptotic}, or  models with a columnar defect \cite{basu2014last,ahlberg2016inhomogeneous}). In addition, 
it has two flat edges near $0$ but there is a point $\alpha_c$ separating  
the linear behaviour from nonlinear one. This point $\alpha_c$ can be viewed as a point of phase transition.

Additionally, the product structure
of the random potential allows to study separately the statistical properties with respect to the processes $F$ and $B$. For example,
it is interesting to study statistics with respect to $B$ for a fixed typical realization of the space component $F$. In this case the problem of finding action-minimizing paths has similarity with problems in the theory of stochastic optimal control. We believe that studying the rich
structure corresponding to the product-type random environment is an interesting problem on its own.

The paper has the following structure.
In Section~\ref{secresults}, we formulate the main results and introduce a few notations.
In Section~\ref{secbasic} and Section~\ref{seclinear}, the results related to the behavior of the shape function are proven.
Section~\ref{secpoisson} contains the proof of convergence to the Poisson process for
the rescaled discrepancies. 
In Section~\ref{sec:properties_of_the_maximizing_path}, we prove the properties the optimal path with free endpoint listed in Theorem~\ref{thm:free}.  
In Section~\ref{sec:distributional_convergence}, we prove conditional result on the limiting distribution for the rescaled optimal action.
Finally, in Section~\ref{secopenproblems}, we discuss open problems and related conjectures.

\section{Formulation of the main results}\label{secresults}
Throughout the paper, we work with a probability space
$(\Omega,\Fc,\Pp)$ supporting 
random variables  $F=(F(x))_{x \in \Z}$ and $B=(B(i))_{i\in\Z}$ with properties described in the Introduction.

Recalling the definition of the action $A_n$ in \eqref{eq:action-product}, 
for $n\in\N$ and $k\in\{-n,-n+1,\ldots,n\}$, we denote
$$\newbar{A}(B,F; n,k) = \min_{\gamma_n: \gamma_n(n)=k} A_n(B,F;\gamma_n).$$

One of the main objects in this paper is the shape function characterizing the rate of linear growth of $\newbar{A}(B,F; n,k)$ as $n\to\infty$ when the endpoint 
$(n,k)=(n,x_n)=(n,[\alpha n])$ is taken to infinity along a line with slope $\alpha\in[-1,1]$. Here $[\cdot]$ denotes the integer part. The existence of such linear rates is standard in the theory of disordered systems. More precisely, the following result is an immediate consequence of Kingman's subadditive ergodic theorem.

\begin{theorem}\label{shape function} There is a continuous, even and concave deterministic function $\shape:[-1,1]\to\R_+$  and an event $\Omega_0\in\Fc$ with $\Prob{\Omega_0}=1$ such that for all $(B,F)\in\Omega_0$ and all $\alpha\in[-1,1]$, we have
\begin{equation}
\label{eq:def_shape_f}
\frac{\newbar{A}(B,F;n,[\alpha n])}{n} \to -\shape(\alpha).
\end{equation}
\end{theorem}
We prove this theorem in Section~\ref{sec:shape-existence}.
The function $\shape$ provided by this theorem is called the shape function. The minus sign in~\eqref{eq:def_shape_f} is chosen to ensure that $\Lambda$ is nonnegative.

It is easy to see that any action-minimizing path tends to stay for a long time at the endpoints of special edges $\{x,x+1\}, \, x \in \Z^1$.
Such special edges have the following property: $F(x)$ and 
$F(x+1)$ are close to $c$ and $-c$, and have different signs. The path
$\gamma$ visits points $x,\, x+1$ intermittently choosing one of them at each time $i$ according to the sign of $B(i)$.

\begin{definition}
\label{def:discrepancy}
For each $x \in \Z$, we define the {\it discrepancy} of the edge $\{x , x + 1\}$ to be
		\begin{equation}
		\label{eq:discrepancy_definition} 
			d(x) = 2c - \abs{\spacevar{x+1} - \spacevar{x}}.
		\end{equation}
\end{definition}

Since there are edges with arbitrary small discrepancy, it is easy to 
see that $\shape(0)=c$. %Indeed, one can always find a path whose action is arbitrary close to $-cn$. 
On the other hand, $\shape(1)=0$ since for $\alpha=1$ there exists only one path $\gamma_n$ and its action is of the order of $n^{1/2}$. Our next result shows that $\shape(\alpha)$ has a corner at $0$ and that, its restriction to $[0,1]$ is a concave
nonlinear function.

\begin{theorem}\label{shape function non-linear}
The following relations hold for the shape function $\shape$:
\begin{align}
\shape(0)&=c,\label{eq:shape-at-0}\\
\shape(1)&=\shape(-1)=0, \label{eq:shape-at-1}\\
\shape(\alpha)&> c(1-|\alpha|), \quad 0<|\alpha| < 1.\label{eq:nonlinearity}
\end{align}
In addition, there is a constant $K>0$ such that
	\begin{align}
		\label{eq:rect_shape_bound1}
		\shape(\alpha) \le c - K|\alpha|,\quad \alpha\in [-1,1],
	\end{align}
 i.e., $\shape$ has a corner at  $\alpha=0$.
\end{theorem}
Relations~\eqref{eq:shape-at-0}, \eqref{eq:shape-at-1}, \eqref{eq:rect_shape_bound1}
are established in Section~\ref{sec:shape-existence}. We prove the nonlinearity property~\eqref{eq:nonlinearity} in Section~\ref{sec:nonlinearity}.

However, for small $\alpha$, the shape function is linear on both sides of $0$. Namely, there is an interval around 
$\alpha=0$ on which identity holds in place of inequality~\eqref{eq:rect_shape_bound1}.
\begin{theorem}\label{shape function linear}
Assume that the exponent $\kappa$ in \eqref{p(x)} is positive.
Then, there exists $\alpha_0\in(0,1)$ such that
%$$\shape(\alpha)=c - K\alpha\quad \text{ for }  \alpha\in[0, \alpha_0],$$
\[
\shape(\alpha)=c - K|\alpha|,\quad \alpha\in[-\alpha_0, \alpha_0],
\]
where $K=(c-\shape(\alpha_0))/\alpha_0$.
\end{theorem}
We prove this theorem in Section~\ref{seclinear}.

\begin{remark}
Theorem~\ref{shape function linear} is the only result where we need to assume that $\kappa>0$. All other results are true for all $\kappa>-1$. We expect that condition $\kappa>0$ can be extended for $\kappa> -\kappa_0$ for some $0<\kappa_0<1$. 
However, we conjecture that for $\kappa$ close to $-1$ the shape function does not have a linear piece (see discussion in Section 8).
\end{remark}

\begin{figure}
        \centering
    \includegraphics[width=5in]{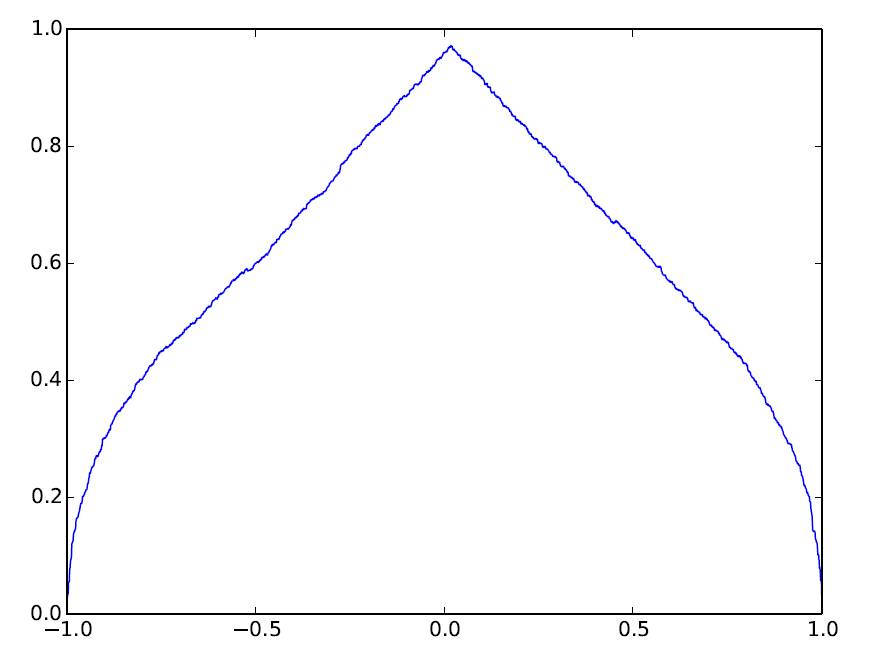}
        \caption{A simulation of the shape function $\shape$ for $\varrho=\frac{1}2 \mathbbm{1}_{(-1,1)}$. Note that although this density is not covered by Theorem~\ref{shape function linear}, $\shape$ seems to have a linear piece, see Section~\ref{secopenproblems} for further discussion.}
        
\end{figure}

We now discuss the properties of the optimal path $\newbar{\gamma}_n$ with a free endpoint $\newbar{\gamma}_n(n)$.
Above we indicated that with large probability the path $\newbar{\gamma}_n$
will reach a very favourable location, that is the edge $e_n=\{\ell_n,\ell_{n}+1\}$ with a small discrepancy $d_n=d(\ell_n)$, and then will stay at this edge for all remaining time. Let $\tau_n$ be the first time $\newbar{\gamma}_n$ visits $\ell_n$ or $\ell_{n}+1$. To understand how far this favourable location will be located, 
we notice that there are two main contributions increasing the minimal
action above the value $-cn$. The first contribution is coming  from the steps before time $\tau_n$, and it is of the order of $|\ell_n|$.
The second contribution is coming  from the steps after time $\tau_n$, and it is of the order of $nd_n/2$. It turns out that the action will be minimal, when these two contributions are in balance. Thus, the order
of $|\ell_n|$ is determined by the condition 
\begin{equation}
\label{eq:heuristic-rel-ell-d}
|\ell_n|\sim nd_n.    
\end{equation}
It is easy to see that for small $u>0$ and any $x\in \mathbb{Z}$, the probability that $d(x)$ is less than $u$ is of order $u^{2\kappa+2}$, see Lemma~\ref{lem:discrep_density_near_0} for a precise statement. This means that $|\ell_n|\sim d_n^{-2\kappa-2}$, and plugging this into~\eqref{eq:heuristic-rel-ell-d}, we obtain  
\begin{equation}
\label{l_n,d_n}
|\ell_n| \sim  n^\zeta, \ \ \ d_n \sim n^{\zeta-1}, 
\end{equation}
where 
\begin{equation}
\label{eq:def_zeta}    
\zeta=\frac{2\kappa+2}{2\kappa+3}.
\end{equation} 

More precisely, we have the following theorem.

\begin{theorem}\label{thm:free}
 Let
\begin{equation}
\label{eq:def_ell_d_tau}
\ell_n=\argmin_{x\in \range(\newbar{\gamma}_n)} d(x),\quad d_n=d(\ell_n)=\min_{x\in \range(\newbar{\gamma}_n)} d(x),\quad \tau_n=\min\{i\,:\,\newbar{\gamma}_n(i)\in\{\ell_n,\ell_n+1\}\}.
\end{equation}
With probability tending to $1$ as $n\to\infty$, we have $\newbar{\gamma}_n(i)\in\{\ell_n,\ell_n+1\} $ for all $i\ge \tau_n$.

Given any $\varepsilon>0$, there are $0<k_1<k_2$ such that for all large enough $n$, we have
\begin{align}
\Prob{k_1 n^\zeta\le  |\ell_n|\le k_2 n^{\zeta}}&\ge 1-\varepsilon,\label{eq:free1}\\
\Prob{k_1 n^{\zeta-1}\le  d_n\le k_2 n^{\zeta-1}}&\ge 1-\varepsilon,\label{eq:free2}\\
\Prob{k_1 n^{\zeta}\le  {cn+A(B,F,\newbar{\gamma}_n)}\le k_2 n^{\zeta}}&\ge 1-\varepsilon.\label{eq:free3}
\end{align}
\end{theorem}
We prove this result in Section~\ref{sec:properties_of_the_maximizing_path}.

This result formalizing the scaling relationships~\eqref{l_n,d_n} is, in fact, based on a scaling limit for renormalized point processes of discrepancies. Namely,
 for every $n$, we can define
a random point process \[\mu_n=\{(\,n^{-\zeta}x\,,\,  n^{1-\zeta} d(x)\,)\,: \, x\in \Z\}\] on the half-plane $\R\times(0,+\infty)$, and state the following result:

\begin{theorem}\label{thm:convtoPPP}
As $n\to \infty$,  the process $\mu_n$
converges in distribution in the vague topology to the Poisson point process $\mu$ on $\R\times(0,+\infty)$
with the intensity function \[q(x,y)= q^2p_\kappa y^{2\kappa+1} , \quad (x,y)\in \R\times(0,+\infty),\]
where $q$ is defined in \eqref{p(x)}, and
\begin{equation}
\label{eq:p_kappa}
p_\kappa=2\mathfrak{B}(\kappa+1,\kappa+1),    
\end{equation}
 $\mathfrak{B}(\cdot,\cdot)$ being the Euler beta function.
\end{theorem}
We prove this theorem in Section~\ref{secpoisson}.

It follows from Theorem~\ref{thm:free} that we can choose a subsequence of the positive integers such that along this subsequence
\[\frac{cn+A(B,F;\newbar{\gamma}_n)}{n^\zeta}\]
converge weakly to a positive random variable.

Under the assumption below, we can prove that  this is even true without taking a subsequence. Moreover, we can also identify the limiting law. 

\medskip

\noindent
{\bf Assumption:}
 There exists a nonrandom constant $M>1$ such that 
 \begin{equation}
 \label{eq:conjecture-time-to-min-discr}
 \lim_{n\to\infty} \Prob{\tau_n<M|\ell_n|}=1.
 \end{equation}

We conjecture that there is $\kappa_0\in(-1,0)$ such that this assumption holds in our setting iff $\kappa> \kappa_0$, but at the moment we are not able to prove this claim.

\begin{theorem}\label{thm:weakconv}
If the Assumption holds, then there is a constant $\newtau$ such that for all $t\ge 0$, we have
\[\lim_{n\to\infty}\Prob{\frac{cn+ A(B,F;\newbar{\gamma}_n)}{\newtau n^\zeta}\ge t}=\exp\left(-t^{2\kappa+3}\right).\]
\end{theorem}

\begin{remark}
The limiting distribution in Theorem~\ref {thm:weakconv} can
be viewed as a  counterpart of the Tracy--Widom law in our model. Note also that the limiting law is universal and depends only on the exponent $\kappa$.
\end{remark}

Since we believe that the Assumption holds once $\kappa> \kappa_0$, the following result is a conjectural extension of  Theorem~\ref{shape function linear}.  
\begin{theorem}\label{linearunderassumption}
If the Assumption holds, then $\shape$ restricted to the interval $\left[0,\frac{1}M\right]$ is linear.    
\end{theorem}

We prove this result in Section~\ref{sec:distributional_convergence}.
\subsection{Notations}

Given the integers $x_1,x_2$ and $0\le t_1\le t_2$, let $\Gamma((t_1,x_1);(t_2,x_2))$ be the set of lazy walks starting at $x_1$ at time $t_1$ and ending at $x_2$ at time $t_2$, that is, the set of functions $\gamma:\{t_1,t_1+1,\dots,t_2\}\to \mathbb{Z}$ such that $\gamma(t_1)=x_1$, $\gamma(t_2)=x_2$ and $|\gamma(i)-\gamma(i+1)|\le 1$ for all $t_1\le i<t_2$. Given $f:\mathbb{Z}\to [-c,c]$, $b:\mathbb{N}\to \{\pm 1\}$ and  $\gamma\in\Gamma((t_1,x_1);(t_2,x_2))$, the action of $\gamma$ is defined as
\[A(b,f;\gamma)=\sum_{i=t_1+1}^{t_2} b(i)f(\gamma(i)).\]
The minimal possible action is denoted by
\begin{equation}
\label{eq:p2p-min-action}
\newbar{A}(b,f;(t_1,x_1),(t_2,x_2))=\min_{\gamma\in \Gamma((t_1,x_1);(t_2,x_2))} A(b,f;\gamma).
\end{equation}

We write $\Gamma(t,x)$ and $\newbar{A}(b,f;t,x)$ in place of $\Gamma((0,0);(t,x))$ and $\newbar{A}(b,f;(0,0),(t,x))$. Also, if there is no ambiguity, we write $\newbar{A}(t,x)$ and $\newbar{A}((t_1,x_1),(t_2,x_2))$ in place of $\newbar{A}(B,F;t,x)$ and $\newbar{A}(B,F;(t_1,x_1),(t_2,x_2))$, respectively.

We write $\mathbb{E}_B$, when we want to emphasize that the expectation is over the random choice of $B$,  while $F$ is kept fixed.  When we want to emphasize that the expectation is over both $B$ and $F$, we write $\mathbb{E}_{B,F}$.

To make use of the symmetries of our models, we define the integer part function $[\cdot]$ a bit unusual way,  that is,
\[[x]=\begin{cases}
\lfloor x\rfloor&\text{for $x\ge 0$,}\\
-\lfloor -x\rfloor&\text{for $x< 0$.}
\end{cases}\]
Here $\lfloor \cdot\rfloor$ is the floor function. In other words, we always round towards $0$.

We will often omit the integer parts for better readability. 

\section{Basic properties of the shape function}\label{secbasic}
\subsection{The discrepancy of an edge}

\begin{definition}
Consider all the paths that stay on the edge $\{x,x+1\}$. Let $\eta_x$ have the smallest action among these, that is, let
% Define the optimal path restricted to the edge $\{x,x+1\}$ by 
		\begin{equation}
		\label{eq:eta} 
			\eta_x (t) = 
				\begin{cases}
					x &\textrm{ if } \spacevar{x} \leq \spacevar{x+1} \textrm{ and } B(t) = 1, \\ 
					x &\textrm{ if } \spacevar{x} \geq \spacevar{x+1} \textrm{ and } B(t) = -1, \\ 
					x+1  &\textrm{ otherwise.}
				\end{cases}
		\end{equation}

\end{definition}

Note that if $\discr{x} < c$ (see Definition~\ref{def:discrepancy}), then necessarily $\spacevar{x}$ and $ \spacevar{x+1}$ have opposite signs.  Since we will focus on edges with small discrepancies, we will restrict our attention to looking at sites where this holds. 
  We will often use the following naive bound on the action of $\eta_x$. Assuming that $F(x)F(x+1)<0$, we have 
  \begin{equation}
  \label{eq:eta_lower_bound}
    \sum_{i=a+1}^{b} B(i)F(\eta_x(i)) \leq -(b-a)\min\left( |\spacevar{x}|, | \spacevar{x+1}| \right)  \leq (b-a)(\discr{x}-c).
  \end{equation}

For more precise control of the action of $\eta_x$, let us introduce the random variable 
\begin{align*}
 N^+(a,b) &= |\{ i \in [a+1,b]\,:\,B(i)=+1\}|, 
\end{align*}
and then compute the action of $\eta_x$ between $a$ and $b$:
\begin{align}
\sum_{i=a+1}^{b} B(i)F(\eta_x(i))
&= \min\left( \spacevar{x}, \spacevar{x+1}\right) \cdot N^+(a,b) - \max \left( \spacevar{x}, \spacevar{x+1}\right) \cdot (b-a-N^+(a,b)) \nonumber \\
\label{eq:eta_action}
&= \left( b-a \right) \left(  \frac{1}{2}\discr{x}-c \right) + \left(  N^+(a,b) - \frac{b-a}{2} \right)\left( \spacevar{x} + \spacevar{x+1} \right).
\end{align}

Note that $|\spacevar{x} + \spacevar{x+1}|\le d(x)$, so
\begin{equation}\label{eq:eta_action2}
\left|\sum_{i=a+1}^{b} B(i)F(\eta_x(i))-\left( b-a \right) \left(  \frac{1}{2}\discr{x}-c \right)\right|\le \left|  N^+(a,b) - \frac{b-a}{2} \right|d(x).
\end{equation}

And as we would expect, the contribution on a large time scale $T$ coming from a path staying on the edge $\{x, x+1\}$ will be $T\left(  \frac{\discr{x}}{2}-c \right)$, plus fluctuations of order $\sqrt{T}$.  

In the next lemma, we describe the distribution of $d(0)$ under the assumptions in \eqref{p(x)}. 

\begin{lemma}
\label{lem:discrep_density_near_0}
Let $p$ be the density of $d(0)$, then
\begin{equation}\label{eqdpdf}\lim_{u\to 0+} \frac{p(u)}{u^{2\kappa+1} }=p_\kappa q^2 
\end{equation}
and
\begin{equation}\label{eqdcdf}\lim_{u\to 0+} \frac{\Prob{d(0)\le u}}{u^{2\kappa+2}}=\frac{p_\kappa}{2\kappa+2} q^2,
\end{equation}
where $p_\kappa$ is defined in~\eqref{eq:p_kappa}.

\end{lemma}
\begin{proof}
Using the fact $\varrho$ is symmetric across $0$, the density of $F(1)-F(0)$ is given by the convolution $\varrho*\varrho$. Then it follows easily that
\[p(u)=\mathbbm{1}(u\ge 0)2(\varrho*\varrho)(2c-u).\]
Thus, for $u\ge 0$, we have
\begin{multline*}p(u)=2\int_0^u \varrho(c-t)\varrho(c-u+t)\,dt=2(q^2+\varepsilon(u))\int_0^u t^\kappa(u-t)^\kappa\,dt\\=2(q^2+\varepsilon(u)) u^{2\kappa+1} \int_0^1 x^\kappa (1-x)^{\kappa}dx=2(q^2+\varepsilon(u)) u^{2\kappa+1} \mathfrak{B}(\kappa+1,\kappa+1),\end{multline*}
where $\lim_{u\to 0} \varepsilon(u)=0$. Thus, we get \eqref{eqdpdf}. Then \eqref{eqdcdf} follows by integrating \eqref{eqdpdf}. 
\end{proof}
% Recall that $\zeta$ was defined in \eqref{eq:def_zeta}.
% \begin{corollary}\label{corsmalldclose}
%     For any $\varepsilon>0$, there is $K$ such that for all large enough $n$, with probability at least $1-\varepsilon$, there is  $x\in\{0,1,\ldots, \lfloor Kn^{\zeta}\rfloor\}$ such that $d(x)\le n^{\zeta-1}$.
% \end{corollary}
% \begin{proof}
% Let $2m$ be the largest even number smaller than $Kn^{\zeta}$. Note that $d(0),d(2),...,d(2m)$ are independent. Using \eqref{eqdcdf} and the trivial identity
% \begin{equation}\label{zetaidentity}(2\kappa+2)(\zeta-1)=-\zeta,\end{equation} we have $\Prob{d(2i)\le n^{\zeta-1}}\ge Cn^{-\zeta}$ for some $C>0$ not depending on $n$. Therefore, the event described in the statement of the corollary has probability at least
% \[1-\prod_{i=0}^m (1-\Prob{d(2i)\le n^{\zeta-1}})\ge 1-\left(1-Cn^{-\zeta}\right)^{0.5 Kn^{\zeta}}\ge 1-\varepsilon\]
% for large enough $K$.
% \end{proof}

\subsection{Existence of the shape function}\label{sec:shape-existence}

Let 
\[
	\cone = \{(t,x) \in \R^2\,:\,0 \leq |x| \leq t\}.
\]

\begin{lemma}
\label{th:shape_function_exists}

There exists a deterministic function $\rectShape:\cone \to [-c,0]$ and an event $\Omega_1\in\Fc$ with 
$\Prob{\Omega_1}=1$
such that if $(B,F)\in\Omega_1$ and $(t,x)\in\cone$, then 
\begin{align} \label{eq:LLN_shape_function}
	\rectShape(t,x) = \lim_{n \to \infty} n^{-1} \rectAction \left(\round{nt}, \round{nx} \right). 
\end{align}
The function $\rectShape$ is subadditive, convex, homogeneous, and symmetric, i.e., for 
$(t_1, x_1),(t_2, x_2),(t,x)$ in~$\cone$,
\begin{align} \nonumber
	\rectShape(t_1, x_1) + \rectShape \left( t_2, x_2 \right) \ge \rectShape(t_1 + t_2, x_1 + x_2), 
\end{align}
 \begin{align} \nonumber
	r \rectShape(t_1, x_1) + (1-r) \rectShape(t_2, x_2) \ge \rectShape(rt_1 + (1-r)t_2, rx_1 + (1-r)x_2),\quad r\in[0,1],
\end{align}
\begin{align} \label{eq:homogeneity}
	\rectShape(t,x) = k^{-1} \rectShape(kt,kx),\quad k > 0,
\end{align}
\begin{align} \nonumber
	\rectShape(t,-x) = \rectShape(t,x).
\end{align}

\end{lemma}

The basis of the proof is the subadditivity of path actions, which is clear from construction:,  
	\begin{align}\label{eqsuperadditivity}
		\rectAction(t_1, x_1) + \rectAction \left( (t_1, x_1), (t_1 + t_2, x_1 + x_2) \right) \geq \rectAction(t_1 + t_2, x_1 + x_2),\quad \text{ for all } (t_1, x_1), (t_2, x_2) \in \cone \cap \Z^2,
	\end{align}
 where we use the notation defined in~\eqref{eq:p2p-min-action}.
This allows us to use Kingman's subadditive ergodic theorem. See Appendix~\ref{sec:appendix}
for the detailed proof.

Let $\shape(x)=-\rectShape(1,x)$. As a straightforward consequence of Lemma~\ref{th:shape_function_exists}, we obtain the following statement. 

\begin{lemma}\label{lem:shape-exists}
There is an even concave  function $\shape:[-1,1]\to [0,c]$ such that with probability~1, for all $x\in [-1,1]$, 
\begin{equation*}
\shape(x)=-\lim_{n\to\infty} \frac{\newbar{A}(B,F;n,\lfloor nx\rfloor )}n.
\end{equation*}
\end{lemma}

\begin{lemma}\label{lem:shape-continuous}
We have $\shape(1)=\shape(-1)=0$ and $\shape$ is continuous on $[-1,1]$.
\end{lemma}

\begin{proof}
The first claim follows since for $\alpha=1$ there is only one admissible path $\gamma_n$ terminating at point $n$ at time $n$, and its action is of the order of $n^{1/2}$ by the Central Limit Theorem. 

Since $\shape$ is concave, $\shape$ is continuous on $(-1,1)$. It remains to prove
$\lim_{x\to 1} \shape(x)=\shape(1)=0$. 
This will follow once we prove that
\begin{equation}\label{eqboundarycontinuity}
    \shape(1-\varepsilon)\le c\varepsilon+(1-\varepsilon)\lambda(\varepsilon)\text{ for all }\varepsilon<\frac{1}{2e},\end{equation}
where $\lambda(\varepsilon)=4c\sqrt{-\frac{\varepsilon}{1-\varepsilon}\log(\varepsilon)}$. 
By the union bound, we have
\begin{equation}
    \label{eqContUnion}\Probbig{\newbar{A}(n,(1-\varepsilon)n) \le -\varepsilon cn-(1-\varepsilon)\lambda(\varepsilon) n}\le \sum_{\gamma\in\Gamma(n,(1-\varepsilon)n)} \Probbig{{A}(B,F;\gamma) \le -\varepsilon cn-(1-\varepsilon)\lambda(\varepsilon) n}.
\end{equation}

Any $\gamma\in\Gamma(n,(1-\varepsilon)n)$ must contain at least $(1-\varepsilon)n$ up steps. For each 
$v\in\{0,1,\ldots,n\}$, there are ${{n}\choose{v}} 2^{n-v}$ lazy walks of length $n$ starting at $0$ and having exactly $v$ up steps. Of course, not all of these walks will end up at $(1-\varepsilon)n.$ Thus,
using an elementary estimate
\[
{n\choose k} \le \Big(\frac{ne}{k}\Big)^k,\quad k=0,1,\ldots, n,
\]
we obtain that for large enough $n$, we have 
\begin{equation}\label{eqgammacount}|\Gamma(n,(1-\varepsilon)n)|\le \sum_{v=(1-\varepsilon)n}^n {{n}\choose{v}} 2^{n-v}\le 2^{\varepsilon n}n {{n}\choose{\varepsilon n}}\le 2^{\varepsilon n}n\left(\frac{ne}{n\varepsilon}\right)^{\varepsilon n}\le \exp(-2\varepsilon\log(\varepsilon) n),\end{equation}
where at the last inequality, we used the assumption that $\varepsilon<\frac{1}{2e}$.

Any $\gamma\in\Gamma(n,(1-\varepsilon)n)$ visits the sites $\{1,2,\dots,(1-\varepsilon)n\}$ at least once. Let $t_j$ be the first time $\gamma$ visits $j$. Then
\[A(B,F;\gamma)\ge \sum_{i=1}^{(1-\varepsilon)n} F(i)B(t_i)-\varepsilon n c.\]

Here $(F(i)B(t_i))_{i=1}^{(1-\varepsilon)n}$ are i.i.d. random variables, $\mathbb{E}F(1)B(t_1)=0$, and $|F(1)B(t_1)|\le c$. Thus,  by Hoeffding's inequality, we have
\begin{equation}\label{eqContHoeffding}\Probbig{A(B,F;\gamma)\le -\varepsilon nc-(1-\varepsilon)\lambda(\varepsilon) n}\le \Prob{\sum_{i=1}^{(1-\varepsilon)n} F(i)B(t_i)\le -(1-\varepsilon)\lambda(\varepsilon) n }\le \exp\left(-\frac{\lambda^2(\varepsilon) (1-\varepsilon)n}{2c^2}\right).
\end{equation}

Combining \eqref{eqContUnion}, \eqref{eqgammacount} and \eqref{eqContHoeffding}, we obtain that for all $\varepsilon<\frac{1}{2e}$, we have
\[\Probbig{\newbar{A}(n,(1-\varepsilon)n) \le -c\varepsilon n-(1-\varepsilon)\lambda(\varepsilon) n}\le \exp(-2\varepsilon\log(\varepsilon) n)\exp\left(-\frac{\lambda^2(\varepsilon) (1-\varepsilon)n}{2c^2}\right).\]
Since the right hand side converges to $0$ as $n$ tends to $\infty$, \eqref{eqboundarycontinuity} follows. 
\end{proof}

Theorem~\ref{shape function} follows now from Lemmas~\ref{lem:shape-exists} 
and~\ref{lem:shape-continuous}.

% \begin{lemma}\label{lem:shape-at-0}
% We have $\shape(0)=c$.
% \end{lemma}
% \begin{proof}
% By Corollary~\ref{corsmalldclose}, there is $K$ such that for all large enough $n$, with probability at least~$\frac{1}2$, there is $x$ such that  $0\le x<Kn^{\zeta}$ and $d(x)\le n^{\zeta-1}$. 

% On this event, let $\gamma_{n,x}$ move ballistically up to site $x$, optimally stay on the edge $\{x, x+1\}$ (following $\eta_{x}$ as defined in \eqref{eq:eta}) as long as possible, and then move ballistically back down to~$0$.  
% That is, 
% \[
% \gamma_{n,x}(i) = \begin{cases}
% i, & 0 \leq i \leq x, \\
% \eta_{x}(i), & x < i < n-x, \\
% n-i, &  n-x \leq i \leq n.
% \end{cases}
% \]
% If $n$ is large enough, then $x<n/2$, so the definition above makes sense. Using \eqref{eq:eta_lower_bound}, we see that
% \[A(B,F;\gamma)\le (2x+1)c-(n-2x-1)(c-d(x))\le -(1-o(1))nc.\]
% Thus, it follows that $\shape(0)=c$.
% \end{proof}

\begin{lemma}\label{lem:shape-at-0}
We have 
\[\lim_{n\to\infty} \frac{A^*(B,F;n,0)}{n}=-c\quad\text{ almost surely.}\]
Therefore, $\shape(0)=c$.
\end{lemma}
\begin{proof}
Let $\varepsilon>0$. By the law of large numbers, almost surely there is an $x>0$ such that $d(x)\le \varepsilon$. 

For $n>2x$, let $\gamma_{n,x}$ move ballistically up to site $x$, optimally stay on the edge $\{x, x+1\}$ (following $\eta_{x}$ as defined in~\eqref{eq:eta}) as long as possible, and then move ballistically back down to~$0$.  
That is, 
\[
\gamma_{n,x}(i) = \begin{cases}
i, & 0 \leq i \leq x, \\
\eta_{x}(i), & x < i < n-x, \\
n-i, &  n-x \leq i \leq n.
\end{cases}
\]
Using \eqref{eq:eta_lower_bound}, we see that
\[A(B,F;\gamma)\le (2x+1)c-(n-2x-1)(c-d(x))\le -n(c-2\varepsilon)\]
for all large enough $n$. Tending to $\varepsilon$ with $0$, the statement follows.
\end{proof}

\begin{lemma} \label{lem:corner} For all $\alpha\in [-1,1]$,
	\begin{align*}
		\label{eq:rect_shape_bound}
		\shape(\alpha) \leq c - |\alpha|(c-D),
	\end{align*}
	where $D = \Exp |\spacevar{0}|$.
 Thus, $\shape$ has a corner at  $\alpha=0$.
\end{lemma}

\begin{proof}
	By symmetry, we may assume that $\alpha\ge 0$. The path $\newbar{\gamma}=\argmin_{\gamma\in \Gamma(n,nx)} A(B,F,\gamma)$ must visit every site between $0$ and $nx$. For $1\le k\le n\alpha$, let $t_k$ be the first time $\newbar{\gamma}$ visits~$k$. Then   
 \[\newbar{A}(B,F;n,n\alpha)\ge \sum_{k=1}^{n\alpha} B(t_k)F(k)-n(1-|x|)c \ge -\sum_{k=1}^{n\alpha} |F(k)|-n(1-|\alpha|)c. \]
 Taking expectation, we obtain
 \[\mathbb{E}\newbar{A}(B,F;n,n\alpha)\ge- n( c - |\alpha|(c-D)) .\]
 Thus, the statement follows.
\end{proof}

Relations~\eqref{eq:shape-at-0}, \eqref{eq:shape-at-1}, \eqref{eq:rect_shape_bound1} of Theorem~\ref{shape function non-linear} follow from Lemmas \ref{lem:shape-continuous}, \ref{lem:shape-at-0} \ref{lem:corner}. To complete the proof of Theorem~\ref{shape function non-linear}, it remains to check the 
nonlinearity relation~\eqref{eq:nonlinearity}. 

\subsection{Nonlinearity of $\shape$ }\label{sec:nonlinearity}

Let us find a lazy walk $\gamma$ from site $0$ to site $n$ by the following algorithm.

\begin{algorithmic}
\State $\gamma(0):= 0$; \quad $t:=0$
\For{$i=1..n$} 
\If {$F(i-1)<-\frac{3}4 c,\quad F(i)>\frac{3}4 c,\quad B(t+1)=1$}
    \State $\gamma(t+1):=i-1$
    \State $\gamma(t+2):=i$
    \State $t:=t+2$
    
\Else
    \State $\gamma(t+1):=i$
    \State $t:=t+1$
    
\EndIf
\EndFor
\end{algorithmic}

Let us introduce the random variables 
\begin{align*}
T_i&=\min\{k: \gamma(k)= i\},\\
 A_i&=\sum_{j=T_{i-1}+1}^{T_i} F(\gamma(j))B(j)
\end{align*}
and the events
\[\mathcal{B}_i=\left\{F(i-1)<-\frac{3}4 c,\,F(i)>\frac{3}4 c\right\}. \]

Observe that
\[A_i=\begin{cases}
F(i-1)+F(i)&\text{on the event }\mathcal{B}_i\cap\{B(T_{i-1}+1)=1,\ B(T_{i-1}+2)=1\},\\
F(i-1)-F(i)&\text{on the event }\mathcal{B}_i\cap\{B(T_{i-1}+1)=1,\ B(T_{i-1}+2)=-1\},\\
-F(i)&\text{on the event }\mathcal{B}_i\cap\{B(T_{i-1}+1)=-1\}.
\end{cases}\]

% On $\mathcal{B}_i\cap\{B(\tau_{i-1}+1)=1,\ B(\tau_{i-1}+2)=1\}$, we have
% $A_i=F(i-1)+F(i).$

% On $\mathcal{B}_i\cap\{B(\tau_{i-1}+1)=1,\ B(\tau_{i-1}+2)=-1\}$, we have
% $A_i=F(i-1)-F(i).$

% On  $\mathcal{B}_i\cap\{B(\tau_{i-1}+1)=-1\}$, we have
% $A_i=-F(i).$

Note that, conditioned on $F(i-1)$ and $F(i)$, the distribution of $(B(T_{i-1}+1),B(T_{i-1}+2))$ is uniform on $\{-1,1\}^2$.

Thus,
\begin{align*}\mathbb{E}\big(\mathbbm{1}_{\mathcal{B}_i}A_i\,|\,F(i-1),F(i)\big)&=\mathbbm{1}_{\mathcal{B}_i}\left(\frac{1}4(F(i-1)+F(i))+\frac{1}4(F(i-1)-F(i))+\frac{1}2(-F(i))\right)\\&=\mathbbm{1}_{\mathcal{B}_i}\frac{F(i-1)-F(i)}2\le 
-\mathbbm{1}_{\mathcal{B}_i}\frac{3}{4}c.\end{align*}

A similar calculation shows that
\[\mathbb{E}(\mathbbm{1}_{\mathcal{B}_i^c}A_i\,|\,F(i-1),F(i))=0.\]

Therefore, denoting $\probBi=\Prob{\mathcal{B}_i}$ (it does not depend on $i$), we have 
\[\mathbb{E}A_i\le -\probBi\frac{3}{4}c,\]
and
\begin{equation}
\label{eq:beginning-part-action}
\mathbb{E} \sum_{j=1}^{T_n} F(\gamma(j))B(j)=\sum_{i=1}^n \mathbb{E} A_i\le -n\probBi\frac{3}4c.    
\end{equation}

Note that the random variables $A_i$ indexed only by even values of $i$ are mutually independent. The same is true for the odd values of $i$. Thus, by Hoeffding's inequality, for all fixed $\varepsilon>0$,  we have
\[ \sum_{j=1}^{T_n} F(\gamma(j))B(j)\le -\left(\frac{3}{4}\probBi-\varepsilon\right) c n\]
with high probability.

Note that $T_i-T_{i-1}=1$ unless $F(i-1)<-\frac{3}4 c$, $F(i)>\frac{3}4 c$, $B(T_{i-1}+1)=1$,  when $T_i-T_{i-1}=2$. Thus,

\[\mathbb{E} T_n=\sum_{i=1}^n \mathbb{E}(T_{i}-T_{i-1})=\left(1+\frac{\probBi}2\right)n.\]

Applying  Hoeffding's inequality the same way as above, we see that if $\varepsilon>0$ fixed, then  with high probability $T_n<\ell n$, for $\ell=1+\frac{\probBi}2+\varepsilon$. On this event, let us define  $\gamma(t)=n$ for all $T_n<t\le \ell n$. From the law of large numbers $|\sum_{t=T_n+1}^{\ell n} F(\gamma(t))B(t)|\le \varepsilon cn$ with high probability. Combining this with~\eqref{eq:beginning-part-action}, we obtain that with high probability there is a walk of length $\ell n$ and 
action at most $-\delta c n$, where $\delta=\frac{3}{4}\probBi-2\varepsilon$. 

It follows that 
\[\shape\left(\frac{1}\ell\right)\ge \frac{\delta c}{\ell}>\left(1-\frac{1}\ell\right)c=\left(1-\frac{1}\ell\right)\shape(0)+\frac{1}\ell \shape(1),\]
where the inequality in the middle is true, if we choose $\varepsilon>0$ to satisfy $\varepsilon<\frac{1}{12}\probBi$ guaranteeing
\[\frac{\frac{3}{4}\probBi-2\varepsilon}{1+\frac{\probBi}2+\varepsilon}>\frac{\frac{\probBi}2+\varepsilon}{1+\frac{\probBi}2+\varepsilon}.\]

This shows that $\shape$ is not a linear function on $[0,1]$, completing the proof of Theorem~\ref{shape function non-linear}. \epf

\section{The existence of a linear segment of $\shape$ -- The proof of Theorem~\ref{shape function linear}}\label{seclinear}

By replacing $\kappa$ with a smaller number, we may assume that there is $\kappa\in\left(0,\frac{1}3\right)$ such that
\begin{equation}
\label{p(x)2}
\varrho(u)\le \kappa^{-1}(c-u)^{\kappa}\text{ for all }u\in[0,c].
\end{equation}
In this section, we will use this assumption in place of \eqref{p(x)}.
\subsection{The main idea of the proof}

 Let $\ell>1$, and let $\newbar{\gamma}=\argmin_{\gamma\in \Gamma(\ell n,n)} A(B,F;\gamma)$. Assume that for $1\le t_0<\ell$ we can find a (random) sequence $0=u_0<u_1<\cdots<u_{t_0n}\le \ell n$ such that $\gamma$ defined by $\gamma(t)=\newbar{\gamma}(u_t)$ 
 is again a lazy walk from $0$ to $n$ of length $t_0n$, in other words, $\gamma$ is obtained from $\newbar{\gamma}$ by removing a few loops.
 Note that 
 \begin{equation}\label{eqheur1}\sum_{j=1}^{t_0n} B(u_j)F(\gamma(j))=\sum_{j=1}^{t_0n} B(u_j)F(\newbar{\gamma}(u_j))\le \sum_{i=1}^{\ell n} B(i)F(\newbar{\gamma}(i))+(\ell-t_0)nc.
 \end{equation}
Assume that the distribution of the random environment $\Phi'(t,x)=B(u_{t})F(x)$ is the same as the distribution of the environment $\Phi(t,x)=B(t)F(x)$, or at least close to it. As the choice of $u_0,u_1,\dots$ might depend on the $B$ and $F$, it is not at all clear that this assumption is justified.   
Under the assumption above, \eqref{eqheur1} implies that
\[\mathbb{E}_{B,F}\newbar{A}(B,F;t_0 n,n)\le \mathbb{E}_{B,F}\newbar{A}(B,F;\ell n,n)+(\ell-t_0)nc.\]
Dividing by $n$ and taking limit, one would obtain that
\[t_0\shape\left(\frac{1}{t_0}\right)\ge \ell \shape\left(\frac{1}{\ell}\right)-(\ell-t_0)c,\]
that is,
\[
\shape\left(\frac{1}{\ell}\right)\le \frac{t_0}{\ell}\shape\left(\frac{1}{t_0}\right)+
\left(1-\frac{t_0}{\ell}\right)c,
\]
which combined with the concavity of $\shape$ and the fact that $\shape(0)=c$ would imply that $\shape$ is linear on $\left[0,\frac{1}{t_0}\right]$.

Thus, the main difficulty in the proof is that distribution of $\Phi$ and $\Phi'$ might not be the same. To overcome this difficulty, one tool will be the concentration result presented in  Section~\ref{secconc}.  Another useful fact is that in the estimate \eqref{eqheur1}, the term $(\ell-t_0)nc$ can be replaced with a somewhat smaller term if we have some additional information on what sites the deleted loops visited. To control this term, we will introduce a modified discrepancy parameter $d^*$ in Section~\ref{secweight}. The significance of our assumptions on the density $\varrho$ is that it provides an upper bound on the number of sites with low discrepancy. We also need to be able to decompose the walk $\newbar{\gamma}$ into loops in a way that we have a good understanding of the discrepancies of the sites visited by these loops. We present our loop decomposition algorithm in Section~\ref{secloops}. 

To carry out the proof outlined above, we need to choose $\ell$ to be large enough.

\subsection{Choosing the weights $f$}\label{secweight}

As simple corollary of Theorem~\ref{shape function}, we have
\begin{lemma}\label{lemma2}
For any $t\ge 1$, we have
\[\lim_{n\to\infty} \frac{\newbar{A}(B,F;\lceil tn\rceil ,n )}n=-t \shape\left(\frac{1}t\right) \text{ almost surely.}\]
\end{lemma}

\begin{lemma}\label{lemma3}
For $\mathbb{P}_F$-almost all choices of $f:\mathbb{Z}\to [-c,c]$, we have that
 if $t_n\ge 1$ is a sequence converging to a limit $t$, then
\begin{equation}
    \label{eq:conv_to_-ta(1t)0}
    \lim_{n\to\infty} \frac{\newbar{A}(B,{f};t_n n,n )}n=-t \shape\left(\frac{1}t\right)\qquad\text{ $\mathbb{P}_B$-almost surely},
    \end{equation}
and
 \begin{equation}
    \label{eq:conv_to_-ta(1t)}
    \lim_{n\to\infty} \frac{\mathbb{E}_{B}\newbar{A}(B,{f};t_n n,n )}n=-t \shape\left(\frac{1}t\right).
    \end{equation}
\end{lemma}
\begin{proof}
It follows from Lemma~\ref{lemma2} that for almost all choices of $f$, we have that for any rational $t\ge 1$, we have
 \begin{equation}\label{eqrac}\lim_{n\to\infty} \frac{\newbar{A}(B,f;t n,n )}n=-t \shape\left(\frac{1}t\right)\qquad\text{ $\mathbb{P}_B$-almost surely.}\end{equation}
 It is easy to derive~\eqref{eq:conv_to_-ta(1t)0}
from \eqref{eqrac} noting that the map $t\mapsto \newbar{A}(B,{f};t,n )-tc$ is monotone decreasing. Using the dominated convergence theorem, \eqref{eq:conv_to_-ta(1t)} follows from \eqref{eq:conv_to_-ta(1t)0}.
\end{proof}

Given a function $f:\mathbb{Z}\to [-c,c]$, we define

\[d_f^*(x)=2c+(-f(x)+\min(f(x-1),f(x),f(x+1))).\]

Since $F(x-1),F(x)$ and $F(x+1)$ are independent random variables with density $\varrho$, we have that for all $h\in[0,2c]$, 
\[
    \Prob{d_F^*(x)\le h}\le \Prob{F(x)\le -c+h}\Prob{F(x-1)\ge c-h}+\Prob{F(x)\le -c+h}\Prob{F(x+1)\ge c-h}.
\]

Using our condition \eqref{p(x)2} on the density $\varrho$, we see that
\[\Prob{F(x)\le -c+h}\le \int_0^h \kappa^{-1}u^\kappa\,du=\kappa^{-1}(1+\kappa)^{-1}h^{\kappa+1}.\]
The same estimate holds for the probabilities $\Prob{F(x-1)\ge c-h}$ and $\Prob{F(x+1)\ge c-h}$. Thus,
\[\Prob{d_F^*(x)\le h}\le 2\kappa^{-2}(1+\kappa)^{-2}h^{2(\kappa+1)}.\]

Combining this estimate with the law of large numbers and Lemma~\ref{lemma3}, we get the following lemma.

\begin{lemma}\label{lemmafchoice}
We can choose ${f}:\mathbb{Z}\to [-c,c]$ such that 
\begin{enumerate}
    \item If $t_n\ge 1$ is a sequence converging to a limit $t$, then
    \[\lim_{n\to\infty} \frac{\mathbb{E}_B\newbar{A}(B,{f};t_n n,n )}n=-t \shape\left(\frac{1}t\right).\]
    \item There is a constant $c_0$ such that for each positive integer $k$, we have
    \[\frac{|\{x\in\{1,2,\dots,n-1\}\,:\,d_f^*(x)\le \frac{2c}k\}|}{n}\le c_0 k^{-2(1+\kappa)} \]
    for all large enough $n$. 
\end{enumerate}
\end{lemma}

We will work with this deterministic $f$ through out the rest of the section. We will also drop the index $f$ from the notation $d_f^*$ and simply write $d^*$.

\subsection{Concentration and choosing the signs $b_n$}\label{secconc}

\begin{lemma}\label{lemmaMcDiarmid}
For any $\varepsilon\ge 0$, $n\in \mathbb{N}$ and $t\ge 1$ such that $tn\in\mathbb{N}$, we have
\[\ProbBig{|\newbar{A} (B,f;tn,n)-\mathbb{E}_B\newbar{A}(B,f;tn,n)|\ge \varepsilon n}\le 2\exp\left(-\frac{\varepsilon^2n}{2c^2t}\right).\]
\end{lemma}
\begin{proof}
 Let $b_1(1),\dots,b_1(tn)$ and $b_2(1),\dots,b_2(tn)$ be two $\pm 1$ sequences such that they only differ at one component.  Let $\newbar{\gamma}$ be $\argmin_{\gamma\in \Gamma(tn,n)} A(b_1,f;\gamma)$. Then
\[\newbar{A} (b_2,f;tn,n)\le A(b_2,f;\newbar{\gamma})\le A(b_1,f;\newbar{\gamma})+2c=\newbar{A}(b_1,f;tn,n)+2c.\]
Similarly, $\newbar{A}(b_1,f;tn,n)\le \newbar{A}(b_2,f;tn,n)+2c$. Thus, $|\newbar{A} (b_1,f;tn,n)- \newbar{A} (b_2,f;tn,n)|\le 2c$. The statement follows from McDiarmid's inequality.  
\end{proof}

Let $U=\{u_1<u_2<\dots\}$ be a subset of the positive integers. We define
\[P_U b(\tj)=b(u_\tj)\text{ and }\]
\[P_U \gamma (0)=\gamma(0) \text{ and } P_U \gamma(\tj)=\gamma(u_\tj)\quad (\text{for }\tj>0).\]

We will use the notation $P_U^+ \gamma$ for the restriction of $P_U \gamma$ to the positive integers.

Let $\ell>1$ be a large constant to be determined later.

Let $\mathcal{U}_{m}$ be all the subsets of $\{1,\dots,\ell  n\}$ of size at least $n$ which are of the form 
\[\{1,2,\dots,\ell n\}\backslash \cup_{i=0}^m (a_i,z_i]\]
for some disjoint subintervals $(a_i,z_i]$ of $\{1,\dots, \ell n\}$. Since we may assume that the sequences $a_0,a_1,\dots a_m$ and $z_0,z_1,\dots z_m$ are both increasing, $|\mathcal{U}_{m}|$ is at most the number of $2(m+1)$-tuples $(a_0,z_0,\dots,a_m,z_m)$ such that $0\le a_0<a_1<\cdots<a_m\le \ell n$ and $0\le z_0<z_1<\cdots<z_m\le \ell n$. Thus, $|\mathcal{U}_{m}|\le {{\ell n+1}\choose{m+1}}^2$.

\begin{lemma}\label{lemmaunionconc}
There is a $c_\kappa>0$ (depending on $\kappa$ and the size of the support of $\varrho$, but not depending on $\ell$) with the following property.  For each $n$ consider the event that for all $0\le m\le n$ and  $U\in  \mathcal{U}_m$, we have
\[\Big|\newbar{A}(P_U B,{f}; |U|,n)-\mathbb{E}_B\newbar{A}( B,{f};|U|,n)\Big|< c_\kappa \left(\frac{m}{ n}\right)^{1/2-\kappa/4} \ell^{1/2+\kappa/4}n.\]
The probability of this event tends to $1$ as $n$ goes to infinity.
\end{lemma}
\begin{proof}
Using the Stirling approximation, we have a $c_0$ such that  for all $1\le m\le n+1$ 
\[|\mathcal{U}_{m}|\le {{\ell n+1}\choose{m+1}}^2=\left(\frac{\ell n+1}{m+1}\right)^2{{\ell n}\choose{m}}^2\le c_0(\ell n+1)^2\exp\left(-2\ell n\left(\frac{m}{\ell n}\log \left(\frac{m}{\ell n} \right)+\left(1-\frac{m}{\ell n}\right)\log\left(1-\frac{m}{\ell n}\right)\right)\right).
\]
We can find a $c_\kappa$ such that 
\[-x\log (x)-(1-x)\log(1-x)\le \frac{c_\kappa^2}{8c^2} x^{1-\kappa/2}\]
for all $x\in [0,1]$. Thus,
\begin{equation}\label{estUm}
    |\mathcal{U}_{m}|\le c_0(\ell n+1)^2\exp\left(\frac{c_\kappa^2}{4c^2}\left(\frac{m}{n}\right)^{1-\kappa/2} \ell^{\kappa/2} n\right).
\end{equation}

Set $\varepsilon= c_\kappa\left(\frac{m}{n}\right)^{1/2-\kappa/4}\ell^{1/2+\kappa/4}$.  For $U\in \mathcal{U}_m$, let $t=\frac{|U|}{n}$.  Applying Lemma~\ref{lemmaMcDiarmid} and using that $t\le \ell$, we have

\begin{align}\label{estP}\Probbig{|\newbar{A} (P_U B, {f},|U|,n)-\mathbb{E}\newbar{A}(P_U B,{f};|U|,n)|\ge \varepsilon n}&\le 2\exp\left(-\frac{\varepsilon^2n}{2c^2t}\right)\\
\notag
&\le 2\exp\left(-\frac{\varepsilon^2n}{2c^2\ell}\right)
\\
\notag
&=2\exp\left(-\frac{c_\kappa^2}{2c^2} \left(\frac{m}{n}\right)^{1-\kappa/2} \ell^{\kappa/2}n\right).
\end{align}

Combining \eqref{estUm}, \eqref{estP} with the union bound, we get that the probability that
\[\Big|\newbar{A} (P_U B, {f},|U|,n)-\mathbb{E}\newbar{A}(P_U B,{f};|U|,n)\Big|\ge c_\kappa \left(\frac{m}{n}\right)^{1/2-\kappa/4} \ell^{1/2+\kappa/4}n\]
for some $U\in  \mathcal{U}_m$ is at most
\[ 2c_0(\ell n+1)^2\exp\left(-\frac{c_\kappa^2}{4c^2}\left(\frac{m}{n}\right)^{1-\kappa/2} \ell^{\kappa/2} n\right).\]

Thus, the probability of the complement of event in the statement of the lemma is at most
\[2c_0(n+1)(\ell n+1)^2\exp\left(-\frac{c_\kappa^2}{4c^2} \ell^{\kappa/2} n^{\kappa/2}\right),\]
which tends to $0$.
\end{proof}

Let ${f}:\mathbb{Z}\to [-c,c]$ be the function provided by Lemma~\ref{lemmafchoice}.
\begin{lemma}\label{lemmabnchoice}
Consider any $\ell>1$. For each $n$, we can choose $b_n:\{1,2,\dots,\ell n\}\to \{-1,+1\}$ such that 
\begin{enumerate}
\item \[\lim_{n\to\infty} \frac{\newbar{A}(b_n,{f};\ell n,n)}{n} =-\ell\shape\left(\frac{1}\ell\right).\]

\item For all large enough $n$, for all $0\le m\le n$ and  $U\in  \mathcal{U}_m$, we have
\[\Big|\newbar{A}(P_U b_n,{f}; |U|,n)-\mathbb{E}_B\newbar{A}(B,{f};|U|,n)\Big|< c_\kappa \left(\frac{m}{ n}\right)^{1/2-\kappa/4} \ell^{1/2+\kappa/4}n.\]

\item For all large enough $n$, we have
\[|\{t\in \{1,2,\dots,\ell n-1\}\,:\, b_n(t)=-1,\,b_n(t+1)=+1\}|\ge \frac{\ell n}{5}.\]
\end{enumerate}
\end{lemma}
\begin{proof}
By the choice of $f$, we have \[\lim_{n\to\infty} \mathbb{E}_B\frac{\newbar{A}(B,{f};\ell n,n)}{n} =-\ell\shape\left(\frac{1}\ell\right).\]
Combining this with Lemma~\ref{lemmaMcDiarmid}, Lemma~\ref{lemmaunionconc} and the law of large numbers, we can easily obtain the statement. 
\end{proof}

Throughout the rest of the proof of Theorem~\ref{shape function linear}, we will assume that for every $\ell>1$, the sequence of vectors $b_n,$ $n\in\N,$ is fixed and satisfies the properties described in 
Lemma~\ref{lemmabnchoice}.  

\subsection{Loop decomposition}\label{secloops}
In this section, we construct the loop decomposition of  $\newbar{\gamma}=\newbar{\gamma}_n=\argmin_{\gamma\in \Gamma(\ell n,n)} A(b_n,f;\gamma)$. 
The whole decomposition and its various elements depend on $\ell$  and $n$, but we suppress this dependence in the notations.

Let $v_0=0$, $v_1=n$, and let $v_2,v_3,\dots,v_n$ be the set $\{1,2,\dots,n-1\}$ ordered so that $d^{*}(v_i)$ is increasing in $i$.

Let $U_{-1}=\{1,\dots,\ell n\}$.

For $i=0,1,2,\dots,n$, let us define
\begin{align*}
a_i&=\min \{t\in U_{i-1}\cup\{0\}\,:\,\newbar{\gamma}(t)=v_i\},\\
z_i&=\max \{t\in U_{i-1}\cup\{0\}\,:\,\newbar{\gamma}(t)=v_i\},\\
L_i&=U_{i-1}\cap (a_i,z_i],\\
U_i&=U_{i-1}\backslash L_{i},\\
e_i&=|\{\{i,i+1\}\subset L_i\,:\,b_n(i)=-1,b_n(i+1)=+1\}|,\\
s_i&=\sum_{j\in L_i} b_n(j){f}(\gamma(j)).
\end{align*}

The proof of the next lemma is straightforward.

\begin{lemma}\hfill\label{lemmacomb}
\begin{enumerate}[(1)]
\item The sets $L_0,L_1,\dots,L_n$ and $U_n$ form a partition of $\{1,\dots,\ell n\}$.
\item \label{item2}$P_{U_i}\newbar{\gamma}$ is a lazy walk from $0$ to $n$, which visits the sites $v_0,v_1,v_2,\dots,v_i$ exactly once. In particular, $|U_i|\ge n$.
\item $L_i=(a_i,z_i]$, so $L_i$ is an interval.
\item \label{UminUmp1}$U_m\in \mathcal{U}_{m}$.
\item For $i\ge 1$, $P_{U_i}\newbar{\gamma}$ only visits the sites $0,1,\dots,n$.
\item For $i\ge 2$, $P_{L_i}^+\newbar{\gamma}$ only visits sites from the set $\{v_{i},v_{i+1},\dots,v_n\}$. 
\item For $i\ge 2$, if $P_{L_i}^+\newbar{\gamma}$ visits a site $x$, then $d^*(x)\ge d^*(v_i)$.
\item $s_i\ge -c|L_i|$.
\item For $i\ge 2$, we have $s_i\ge -c|L_i|+e_i d^*(v_i)$.
\item \label{item8}$\sum_{i=0}^n e_i\ge |\{i\in \{0,1,\dots,\ell n-1\}\,:\,b_n(i)=-1,b_n(i+1)=+1\}|-2(n+1)$.
\item \label{item9}  For all $0\le m\le n$, we have
\begin{align*} 
\newbar{A}\big(P_{U_m} b_n,{f};|U_m|,n\big)&= \newbar{A}(b_n,{f};\ell n,n)-\sum_{j=0}^m s_j\\&\le\newbar{A}(b_n,{f};\ell n,n)+ (\ell n-|U_m|)c-\sum_{j=2}^m e_j d^*(v_j).
\end{align*}
\end{enumerate}
\end{lemma}

\subsection{Finishing the proof of Theorem~\ref{shape function linear}}

\begin{lemma}\label{lemmasupd}
For any $\ell>1$, we define $v_0,\ldots v_n$ as in Section~\ref{secloops}. If a sequence $(m_n)$ satisfies $2\le m_n\le n$ and $\limsup_{n\to\infty} \frac{m_n}{n}=\delta$, then
\begin{equation}
\label{eq:sum_ej_d}
\limsup_{n\to\infty} \frac{\sum_{j=2}^{m_n} e_j d^*(v_j)}n\le c_\kappa \delta^{1/2-\kappa/4} \ell^{1/2+\kappa/4}.
\end{equation}
Here the constant $c_\kappa$ does not depend on $\ell$.
\end{lemma}
\begin{proof}

Combining the choice of $b_n$  provided by Lemma~\ref{lemmabnchoice} with \eqref{item9} and \eqref{UminUmp1} of Lemma~\ref{lemmacomb}, we obtain that there is $c_\kappa>0$ with the following property.  For all large enough $n$, we have
\[\mathbb{E}_B\big(\newbar{A}(B,{f};|U_{m_n}|,n)\big)-c_\kappa \left(\frac{m_n}{ n}\right)^{1/2-\kappa/4} \ell^{1/2+\kappa/4}n\le\newbar{A}(b_n,{f};\ell n,n) +(\ell n-|U_{m_n}|)c-\sum_{j=2}^{m_n} e_j d^*(v_j),\]
see Section~\ref{secloops} for the definition of sets $U_m$, $0\le m\le n$.

We may assume that  $\lim_{n\to\infty}\frac{|U_{m_n}|}{n}=t$ exist.   If not, we may take a convergent subsequence.   We have $t\ge 1$ due to (\ref{item2}) of Lemma~\ref{lemmacomb}. Dividing by $-n$ and taking $n\to\infty$, we obtain 
\begin{equation}t\shape\left(\frac{1}t\right)+c_\kappa \delta^{1/2-\kappa/4} \ell^{1/2+\kappa/4}\ge \ell\shape\left(\frac{1}\ell\right)-
(\ell -t)c+\limsup_{n\to\infty} \frac{\sum_{j=2}^{m_n} e_j d^*(v_j)}n .\label{eqlimit}\end{equation}
Since $\shape$ is concave, we have
\[
\shape\left(\frac{1}{\ell}\right)\ge 
\left(1-\frac{t}{\ell}\right)\shape(0) + \frac{t}{\ell}\shape\left(\frac{1}{t}\right).
\]
Using that $\shape(0)=c$, we can rewrite this as
\[ \ell\shape\left(\frac{1}\ell\right)-
(\ell -t)c\ge t\shape\left(\frac{1}t\right).\]
Inserting this into \eqref{eqlimit}, we obtain~\eqref{eq:sum_ej_d}.
\end{proof}

Let  \[m_{n,k}=\max \left\{2 \le i\le n\,:\, d^*(v_i)\le \frac{2c}k\right\},\] if the set on the right-hand side is nonempty, and let $m_{n,k}=1$ otherwise.  Note that $m_{n,1}=n$. By the choice of $f$ provided by Lemma~\ref{lemmafchoice}, we have 
\[\limsup_{n\to\infty} \frac{m_{n,k}}{n}\le c_0 k^{-2(1+\kappa)}.\]

Combining this with Lemma~\ref{lemmasupd}, we obtain 
\begin{align}
    \limsup_{n\to\infty} \frac{\sum_{j=2}^{m_{n,k}} e_j d^*(v_j)}n&\le c_\kappa \left( c_0 k^{-2(1+\kappa)}\right)^{1/2-\kappa/4} \ell^{1/2+\kappa/4}\nonumber\\&\le c_\kappa'  \ell^{1/2+\kappa/4} k^{-(1+\kappa/3)},\label{eqsumdvj}
\end{align}
where we used our assumption that $\kappa<\frac{1}3$, see \eqref{p(x)2}.

Then for each $k>1$, we have
\[\sum_{j=m_{n,k}+1}^{m_{n,k-1}} e_j d^*(v_j)\ge \frac{2c}{k} \sum_{j=m_{n,k}+1}^{m_{n,k-1}} e_j.\]

For $1\le g<h$, adding these inequalities for $k=g+1,g+2,\dots, h$, we get that

\[\sum_{j=2}^{m_{n,g}} e_j d^*(v_j) \ge \sum_{k=g+1}^{h}\sum_{j=m_{n,k}+1}^{m_{n,k-1}} e_j d^*(v_j)\ge \sum_{k=g+1}^h  \frac{2c}{k} \sum_{j=m_{n,k}+1}^{m_{n,k-1}} e_j.\]

Combining this with \eqref{eqsumdvj}, we obtain
\[\limsup_{n\to\infty} \frac{1}n \sum_{k=g+1}^h  \frac{2c}{k} \sum_{j=m_{n,k}+1}^{m_{n,k-1}} e_j \le c_\kappa'  \ell^{1/2+\kappa/4} g^{-(1+\kappa/3)}.\]

Summing up these inequalities for $g=1,2,\dots,h-1$, we get that
\[\limsup_{n\to\infty} \frac{1}n \sum_{g=1}^{h-1} \sum_{k=g+1}^h  \frac{2c}{k} \sum_{j=m_{n,k}+1}^{m_{n,k-1}} e_j \le \sum_{g=1}^{h-1} c_\kappa'  \ell^{1/2+\kappa/4} g^{-(1+\kappa/3)}.\]

Note that 
\[\sum_{g=1}^{h-1} \sum_{k=g+1}^h  \frac{2c}{k} \sum_{j=m_{n,k}+1}^{m_{n,k-1}} e_j=\sum_{k=2}^h\frac{2c(k-1)}{k} \sum_{j=m_{n,k}+1}^{m_{n,k-1}} e_j \ge c \sum_{j=m_{n,h}+1}^{m_{n,1}} e_j=c \sum_{j=m_{n,h}+1}^{n} e_j.\]

Moreover,  there is $c_\kappa''>0$ such that for all $h$,
\[\sum_{g=1}^{h-1} c_\kappa'  \ell^{1/2+\kappa/4} g^{-(1+\kappa/3)}\le \sum_{g=1}^{\infty} c_\kappa'  \ell^{1/2+\kappa/4} g^{-(1+\kappa/3)} \le c c_\kappa''  \ell^{1/2+\kappa/4}. \]

Thus,
\begin{equation}\label{eqsumei}
    \limsup_{n\to\infty} \frac{1}n\sum_{j=m_{n,h}+1}^{n} e_j\le c_\kappa''  \ell^{1/2+\kappa/4} .
\end{equation}

Note that the right hand side does not depend on $h$. Also the constant $c_\kappa''$ does not depend on $\ell$.

\begin{lemma}
Assume that $\ell$ is large enough. Then we can choose a sequence $h_n$ such that $h_n\to \infty$ and
\[\sum_{j=0}^{m_{n,h_n}} e_j\ge \frac{\ell n}{6}\]
for all large enough $n$, and 
\[\lim_{n\to\infty}\frac{m_{n,h_n}}{n}=0.\]

\end{lemma}
\begin{proof}
Combining the choice of ${f}$ provided by Lemma~\ref{lemmafchoice} and \eqref{eqsumei}, we can choose $h_n$ such that $h_n\to\infty$, $\lim_{n\to\infty}\frac{m_{n,h_n}}{n}=0$ and for all large enough $n$, we have
\[ \sum_{j=m_{n,h}+1}^{n} e_j\le 2nc_\kappa''  \ell^{1/2+\kappa/4}.\]
Then
\begin{align*}\sum_{j=0}^{m_{n,h_n}} e_j&= \sum_{i=0}^n e_i- \sum_{j=m_{n,h}+1}^{n} e_j\\&\ge |\{i\in \{1,\dots,\ell n-1\}\,:\,b_n(i)=-1,b_n(i+1)=+1\}|-2(n+1)-\sum_{j=m_{n,h}+1}^{n} e_j\\&\ge \frac{\ell n}5 -3n-2nc_\kappa''  \ell^{1/2+\kappa/4}\\&\ge \frac{\ell n}6
\end{align*}
for all large enough $n$ provided that $\ell$ is large enough. Here, the first inequality is just part \eqref{item8} of Lemma~\ref{lemmacomb}, and the second inequality follows from the choice of $b_n$ provided by Lemma~\ref{lemmabnchoice}.
\end{proof}

Now we have
\begin{align}
    \mathbb{E}_B\newbar{A}\big(B,{f};|U_{m_{n,h_n}}|,n\big)&\le \newbar{A}\big(P_{U_{m_{n,h_n}}} b_n,{f};|U_{m_{n,h_n}}|,n\big) +c_\kappa \left(\frac{m_{n,h_n}}{ n}\right)^{1/2-\kappa/4} \ell^{1/2+\kappa/4}n\label{eqpen}\\&\le \newbar{A}(b_n,{f};\ell n,n) +c(\ell n-|U_{m_{n,h_n}}|)+c_\kappa \left(\frac{m_{n,h_n}}{ n}\right)^{1/2-\kappa/4} \ell^{1/2+\kappa/4}n\nonumber,
\end{align}
where the first inequality follows from the choice of $b_n$ provided by Lemma~\ref{lemmabnchoice} and \eqref{UminUmp1} of Lemma~\ref{lemmacomb}, and the second inequality follows from \eqref{item9} of Lemma~\ref{lemmacomb}.

Note that 
\[|\{1,2,\dots,\ell n\}\backslash U_{m_{n,h_n}}|=\sum_{i=0}^{m_{n,h_n}} |L_i|\ge \sum_{j=0}^{m_{n,h_n}} e_j\ge \frac{\ell n}{6} \]
for all large enough $n$. Thus, \[\limsup_{n\to\infty} \frac{|U_{m_{n,h_n}}|}n\le \frac{5}6\ell.\] We may assume that 
\[
t=\lim_{n\to\infty} \frac{|U_{m_{n,h_n}}|}n\in [1,\ell)
\] exists. If not, we can consider a convergent subsequence. Thus, dividing \eqref{eqpen} by $-n$ and taking $n\to\infty$, we obtain
\[t\shape\left(\frac{1}t\right)\ge \ell\shape\left(\frac{1}\ell\right)-\shape(0)(\ell-t),\]
or, rearranging,
\[
\shape\left(\frac{1}{\ell}\right)\le \frac{t}{\ell}\shape\left(\frac{1}{t}\right)+
\left(1-\frac{t}{\ell}\right)\shape(0).
\]
Combining this with the concavity of $\shape$, we obtain that $\shape$ is linear on $\left[0,\frac{1}{\ell}\right]\subset \left[0,\frac{1}t \right]$. This completes the proof of Theorem~\ref{shape function linear}. \epf

\section{Poisson approximation -- The proof of Theorem~\ref{thm:convtoPPP}} % (fold)
\label{secpoisson}
\subsection{Basics of Poisson Approximation}
\label{sub:basics_of_poisson_approximation}

To prove Theorem \ref{thm:convtoPPP}, we need some terminology and theory on random point processes that we proceed to describe. This
description is based on~\cite{kallenberg1983random}. We restrict our attention to point processes in $\R^2$ since that is the
class of point processes we deal with in this section. The space~$\Ng$ consists of all integer-valued (nonnegative locally
bounded) measures defined on the set $\Bc$ of bounded Borel sets in $\R^2$. This set is equipped with 
$\sigma$-algebra~$\Nc$ generated by maps $\mu\mapsto\mu(B)$,
$B\in\Bc$. A point process $\mu$ is a measurable map from a probability space $(\Omega,\Fc,\Pp)$ to $(\Ng,\Nc)$. Its
distribution is the pushforward of $\Pp$ under $\mu$. A point process $\mu$ is called a.s.-simple if with
probability 1,
all the atoms of $\mu$ have weight~$1$, which corresponds to the situation where no two points of the point process
coincide. 

The natural topology on $\Ng$ is vague topology. Its base is given by finite intersections of $\Ng$-sets of the form
\[
 \bigg\{\mu: s< \int f\ud\mu < t\bigg\},
\]
where $s,t\in\R$, and $f$ is any continuous function with compact support.

For any random point process $\mu$, we denote \[\Bc_\mu=\{B\in \Bc:\ \mu(\partial B)=0\ \text{a.s.}\},\]
where $\partial B$ denotes the boundary of $B$.

A collection $\Uc$ of bounded Borel sets in $\R^2$ is called a DC(dissecting and covering)-ring if it is a ring such
that for any $B\in\Bc$
and any $\eps>0$, there is a finite cover of $B$ by $\Uc$-sets of diameter less than $\eps$. A DC-semiring is a
semiring with the same property. Recall that $\Uc$ is a ring if it is nonempty and for all $A,B\in \Uc$, we have $A\cap B\in \Uc$ and $A\setminus B\in \Uc$. Moreover, $\Uc$~is a semiring if it is nonempty and for all $A,B\in \Uc$, we have $A\cap B\in \Uc$ and $A\setminus B$ is the disjoint union of finitely many sets from $\Uc$.  An example of a DC-semiring is the collection of all rectangles
$[a_1,a_2)\times[b_1,b_2)$. The collection of all finite unions of rectangles is a DC-ring.

To check the weak convergence to an a.s.-simple point process in vague topology it is essentially
sufficient to check the convergence of the avoidance function that computes the probability that there is no points
inside a given set.
The following theorem is a specific case of Theorem~4.7 in~\cite{kallenberg1983random}.
\begin{theorem} 
\label{th:Kallenberg}
Let $(\mu_n)_{n\in\N}$ and $\mu$ be point processes in $\R^2$ and assume that $\mu$ is a.s.-simple.
Suppose that $\Uc\subset \Bc_\mu$ is a DC-ring and $\Ic\subset \Bc_\mu$ is a DC-semiring. Suppose further that
\begin{equation}
 \lim_{n\to\infty}\Prob{\mu_n(U)=0}=\Prob{\mu(U)=0},\quad U\in \Uc,
\label{eq:conv_avoidance}
\end{equation}
 and
\begin{equation}
 \limsup_{n\to\infty} \E \mu_n(I)\le \E \mu(I)<\infty,\quad I\in \Ic. 
\label{eq:limsup_expect}
\end{equation}

Then $\mu_n$ converges to $\mu$ weakly in vague topology.
\end{theorem}

\subsection{Proof of Theorem \ref{thm:convtoPPP}}

Let us take $\Ic$ to be the semiring of rectangles, $\Uc$ to be the ring of finite unions thereof, and check 
conditions~\eqref{eq:conv_avoidance} and~\eqref{eq:limsup_expect} of Theorem~\ref{th:Kallenberg}.

The identity
\begin{equation}\label{zetaidentity}(2\kappa+2)(\zeta-1)=-\zeta,\end{equation}
follows trivially from \eqref{eq:def_zeta}.

Take a rectangle $I=[a_1,a_2)\times[b_1,b_2)$. Using \eqref{eqdpdf} and \eqref{zetaidentity}, we see that
\begin{align}
\label{eq:expect_I}
 \E \mu_n(I)&=\sum_{n^{\zeta}a_1\le k< n^{\zeta}a_2}\Prob{d(k)\in[b_1n^{\zeta-1},b_2n^{\zeta-1})}\\
 &=\sum_{n^{\zeta}a_1\le k< n^{\zeta}a_2} (1+o(1))\int_{b_1n^{\zeta-1}}^{b_2n^{\zeta-1}} p_\kappa q^2 u^{2\kappa+1}du\nonumber\\&=\sum_{n^{\zeta}a_1\le k< n^{\zeta}a_2} (1+o(1)) \left( n^{\zeta-1}\right)^{2\kappa+2}  \int_{b_1}^{b_2} p_\kappa q^2 y^{2\kappa+1} dy\nonumber\\
 &=(1+o(1)) n^{-\zeta}\sum_{n^{\zeta}a_1\le k< n^{\zeta}a_2}   \int_{b_1}^{b_2} p_\kappa q^2 y^{2\kappa+1} dy\nonumber\\&=(1+o(1))\int_{a_1}^{a_2} \int_{b_1}^{b_2} p_\kappa q^2 y^{2\kappa+1} dy  dt=(1+o(1))\mathbb{E}\mu(I).\nonumber
 \end{align}

Thus, \eqref{eq:limsup_expect} holds true. 

To prove~\eqref{eq:conv_avoidance}, we take arbitrary disjoint rectangles \[U_i=[a_1^{(i)},a_2^{(i)})\times[b_1^{(i)},b_2^{(i)}),\quad i=1,\ldots,m,\]
define
\[
 U=\bigcup_{i=1}^m U_i,
\]
and compute
\[
 \lim_{n\to\infty}\Prob{\mu_n(U)=0}.
\]
We quote Theorem~4.3 from \cite{chen1975poisson}.
\begin{theorem}\label{th:poisson_for_m-dependent:chen} If $X_1,\ldots,X_n$ are $m$-dependent r.v.'s taking values $0$ and $1$, then
for any function $h:\Z_+\to[-1,1]$,
\[
 \left|\E h\left(\sum_{i=1}^n X_i\right)-\sum_{k=0}^{\infty}e^{-\lambda}\frac{\lambda^k}{k!}h(k)\right|\le
6\min\{\lambda^{-1/2},1\}\left[\sum_{i,j: i\ne j}\cov (X_i,X_j)+(4m+1)\sum_{i=1}^n p_i^2\right], 
\]
where
$p_i=\Prob{X_i=1}$, and $\lambda=\sum_{i=1}^n p_i$.
\end{theorem}

We shall apply this theorem to our situation.
Notice that
\[
 \mu_n(U)=\sum_{k} Z_{n,k}
\]
is a finite sum of $1$-dependent r.v.'s  
\[
Z_{n,k}=\mathbbm{1}_{\{(n^{-\zeta}k,\, n^{\zeta-1}d(k))\in U\}},\quad k\in\Z, 
\]
since only $k\in\bigcup_i [n^{\zeta}a_1^{(i)},n^{\zeta}a_2^{(i)})$ contribute to this sum. 
Therefore, we can take $m=1$ in Theorem~\ref{th:poisson_for_m-dependent:chen}. Next,
\[
 \Prob{Z_{n,k}=1}=\Prob{d(k)\in n^{\zeta-1}U(n^{-\zeta}k)},
\]
where $U(x)=\{y:\ (x,y)\in U\}$. The same calculation as in \eqref{eq:expect_I}, gives us 
\[p_{n,k}=\Prob{Z_{n,k}=1}=(1+o(1))n^{-\zeta}\int_{U(n^{-\zeta}k)}   p_\kappa q^2 y^{2\kappa+1}dy,\]
so that
\[
 \lambda_n=\sum_{k\in\Z}\Prob{Z_{n,k}=1}=(1+o(1))\sum_{k}n^{-\zeta}\int_{U(n^{-\zeta}k)} p_\kappa q^2 y^{2\kappa+1} d y\to \lambda,\quad	 n\to\infty,
\]
where
\[
 \lambda=\int_U p_\kappa q^2 y^{2\kappa+1} d t dy.
\]
Note that since $U$ is bounded, we have $p_{n,k}^2=O(n^{-2\zeta})$.

Thus,
\[
\lim_{n\to\infty}\sum_{k\in\Z}p_{n,k}^2=O(n^{-\zeta})=o(1),
\]
since the sum above has $O(n^{\zeta})$ non-zero terms.

Most covariances in the estimate provided by Theorem~\ref{th:poisson_for_m-dependent:chen} are equal to zero in our case. 
The only nontrivial contribution comes from
\begin{align*}
 \cov (Z_{n,k},Z_{n,k+1})&\le \Prob{Z_{n,k}=1,\ Z_{n,k+1}=1}\\
&\le  \Prob{d(k)\in n^{\zeta-1}U(n^{-\zeta}k),\ d(k+1)\in n^{\zeta-1}U(n^{-\zeta}(k+1))}\\
&\le \Prob{d(k)\le n^{\zeta-1}b_*,\ d(k+1)\le n^{\zeta-1}b_*},
\end{align*}
where $b_*=\max_i b_2^{(i)}$.
The right-hand side is bounded by
\begin{align*}
 \Prob{|F(k)|,|F(k+1)|,|F(k+2)|\ge c-n^{\zeta-1}b_*}\le \left(O\left(\left(n^{\zeta-1}\right)^{\kappa+1}\right)\right)^3=O(n^{-\frac{3\kappa+3}{2\kappa+3}}).
\end{align*}

Since there are $O(n^{\zeta})$ indices $k$ contributing nonzero covariances we conclude that the total contribution from
the covariance term is $O(n^{-\frac{\kappa+1}{2\kappa+3}})$. This concludes the proof that the estimate provided by
Theorem~\ref{th:poisson_for_m-dependent:chen} converges to 0 as $n\to\infty$.  We now take $h(k)=\mathbbm{1}_{k=0}$,
so that $\E h(\sum Z_{n,k})$ is exactly the avoidance function for measure $\mu_n$, which completes the proof
of Theorem~\ref{thm:convtoPPP} once we recall the definition \eqref{eq:p_kappa}.\epf

\section{The minimizing path -- The proof of Theorem~\ref{thm:free}} % (fold)
\label{sec:properties_of_the_maximizing_path}

% We will often bound the action of a path in terms of the discrepancies it moves through.  We would like to say that whenever there is a sign change of $B(\cdot)$, we can bound the action on such two sites by $ \min_{x\in \range(\gamma)}\discr{x}-2c$.  But to avoid overcounting such actions (since there can be two sign changes in only three steps), we index pairs only by the even numbers.  

We have the following simple lemma.

\begin{lemma}
  \label{lm:bound_path_by_min_discr}
  For any path $\gamma$ in $\paths((t_1,x_1);(t_2,x_2))$,
  \[
    \action(B,F;\gamma) \ge- c(t_2-t_1) + \sum_{\substack{i \in (t_1,t_2) \\ i \textrm{ even}}} \1_{\left\{ B(i) \neq B(i+1) \right\}} \cdot \min_{x \in \range(\gamma)} \discr{x}.
  \]
\end{lemma}
\begin{proof}
Clearly, $F(\gamma(i))B(i)+F(\gamma(i+1))B(i+1)\ge -2c$. If $B(i)\neq B(i+1)$, then this inequality can be improved to  $F(\gamma(i))B(i)+F(\gamma(i+1))B(i+1)\ge  \min_{x\in \range(\gamma)}\discr{x}-2c$. Thus,
\[F(\gamma(i))B(i)+F(\gamma(i+1))B(i+1)\ge-2c+\1_{\left\{ B(i) \neq B(i+1) \right\}} \cdot \min_{x \in \range(\gamma)} \discr{x}.\]
Summing these inequalities for even $i$, the lemma follows.
\end{proof}

We now prove some basic properties of the optimal path $\maxPath_n$.

The following corollary of Theorem~\ref{thm:convtoPPP} guarantees the existence of good ``candidate sites'' for the best path, while also guaranteeing that such points are well separated with high probability.

\begin{corollary}
\label{cor:separation_of_discrepancies}
Assume that we are given $\varepsilon> 0$ and  $a> 0$. 
\begin{enumerate}[(i)]
\item
Then there is $b>0$ such that for all sufficiently large $n$, we have 
\[\Prob{\mu_n\left( R \right) \geq 1 } > 1 - \varepsilon,\]
where $R = (-a, a) \times (0,b)$.
\item
Let $b>0$. There is $\delta>0$  such that for all sufficiently large $n$, we have
\begin{align*}
	\ProbBig{ \mu_n\left(R_{\delta}(y)\right) \leq 1 \text{\rm\ for all }  y\  \text{\rm such that\ }  R_{\delta}(y)\subset R } > 1 - \varepsilon,
\end{align*}
where $R = (-a, a) \times (0,b)$ and $R_{\delta}(y) = (-a, a )\times (y - \delta, y + \delta)$.

\end{enumerate}
\end{corollary}

\begin{proof}
We start by proving (i). First we prove this fact for the limiting Poisson point process~$\mu$ with driving measure $\nu$ described in Theorem \ref{thm:convtoPPP}.  
Given $a$, for all large enough $b$, we have $\Prob{\mu(R) \geq 1} >   1 - \varepsilon$, which is clear from the density of $\nu$. By Theorem 4.2 in \cite{kallenberg1983random}, vague convergence of $\mu_n$ to $\mu$ implies that $\mu_n(R)$ converge to $\mu(R)$ in law, so the statement follows.

To prove (ii), we also start by considering the limiting process $\mu$. First, we consider the case  $2\kappa+1<0$. We condition on $N_R$, the number of points in the rectangle $R$. Then the points are independent and identically distributed in $R$ with distribution $\nu(\cdot)/\nu(R)$. 
Thus for any two such points, the probability of both of them being within a given rectangle $R_{2 \delta}(y) \subset R$ is 
\begin{equation}\label{kappasmall}\left(\frac{2a}{\nu(R)}\right)^2 \left(\int_{y-2\delta}^{y+2\delta} p_\kappa q^2 u^{2\kappa+1} du\right)^2\le \left(\frac{2a}{\nu(R)}\right)^2 \left(\int_0^{4\delta}p_\kappa q^2 u^{2\kappa+1} du\right)\left(\int_{y-2\delta}^{y+2\delta} p_\kappa q^2 u^{2\kappa+1} du\right). \end{equation}
 
And thus the probability of any two points being within a given rectangle $R_{2 \delta}(y)$ is bounded by \[N_R(N_R-1)\left(\frac{2a}{\nu(R)}\right)^2 \left(\int_0^{4\delta}p_\kappa q^2 u^{2\kappa+1} du\right)\left(\int_{y-2\delta}^{y+2\delta} p_\kappa q^2 u^{2\kappa+1} du\right) .\]  

Now, consider the rectangles $R_{ 2\delta}^{k} := R_{ 2\delta}( 2k\delta)$, where $k\in \Z$ and $R_{2\delta}^k\subset R$. Assuming that $b/(2\delta)$ is an integer,  these cover $R$. If each $R_{ 2\delta}^{k}$ contains no more than $1$ point, then no two points are within any $R_{ \delta}(y)$ for any $R_\delta(y) \subset R$. 

Using a union bound, we obtain
\begin{align*}
  &\ProbBig{ \mu(R_{\delta}(y))  \geq 2  \textrm{ for some } R_\delta(y) \subset R  }  \\
  & \leq \Probbig { \mu(R_{2\delta}^{k}) \geq 2 \text{ for some } k } \\ 
  & \leq  \Exp[ N_R(N_R-1)] \left(\frac{2a}{\nu(R)}\right)^2 \left(\int_0^{4\delta}p_\kappa q^2 u^{2\kappa+1} du\right)\sum_k \int_{2k\delta-2\delta}^{2k\delta+2\delta} p_\kappa q^2 u^{2\kappa+1} du\\
  &\le 2\Exp[ N_R(N_R-1)] \left(\frac{2a}{\nu(R)}\right)^2 \left(\int_0^{4\delta}p_\kappa q^2 u^{2\kappa+1} du\right)\left(\int_0^b  p_\kappa q^2 u^{2\kappa+1} du\right),
\end{align*} 
where in the last line the factor $2$ is coming from the fact that each point of $R$ is covered by at most $2$ rectangles $R_{ 2\delta}^{k}$.

This probability can be made arbitrarily small by taking $\delta$ small enough.  

Finally, by Theorem 4.2 in \cite{kallenberg1983random}, vague convergence of $\mu_n$ to $\mu$ implies that the joint distribution of $\left\{\mu_n \left( R_{2\delta}^{k} \right) \right\}_{R_{2\delta}^k\in R}$ converges to the joint distribution of $\left\{\mu \left( R_{2\delta}^{k} \right) \right\}_{R_{2\delta}^k\in R}$.

The proof of the case $2\kappa+1\ge 0$ is similar, we just need to replace the right hand side of~\eqref{kappasmall} with
\[\left(\frac{2a}{\nu(R)}\right)^2 4\delta p_\kappa q^2 b^{2\kappa+1} \left(\int_{y-2\delta}^{y+2\delta} p_\kappa q^2 u^{2\kappa+1} du\right).\qedhere\]
\end{proof}

  In the rest of this section, we will consider events $\mathcal{A}=\mathcal{A}_{b,k_2,k_1,n}$ which depend on the positive real parameters $b,k_2,k_1$ and the positive integer $n$. 
 The idea of the following definition is to be able to say that if we choose parameters $b,k_2,k_1$
 sequentially, then we can guarantee a bound on the complement of  $\mathcal{A}_{b,k_2,k_1,n}$ for all large enough $n$ for a range of those parameter values. 

 We say that $\mathcal{A}$  has high probability if for every $\varepsilon>0$, we can choose a number $b^*\in(0,\infty)$, and functions 
 $k_2^*:(0,+\infty)\to(0,+\infty)$, $k_1^*:(0,+\infty)^2\to(0,+\infty)$  such that for all $b\in(b^*,+\infty)$, $k_2\in(k_2^*(b),+\infty)$ and 
 $k_1\in(0,k_1^*(b,k_2))$, we have $\Prob{\mathcal{A}}>1-\varepsilon$ for all large enough $n$.

We will denote these events by $\mathcal{A}_1,\mathcal{A}_2,\dots$ and
  define 
  \begin{equation*}
  %\label{eq:I_k}
      \mathcal{D}_k=\bigcap_{i=1}^k \mathcal{A}_i.
  \end{equation*}  
  Note that the intersection of high probability events has also high probability. Thus, if $\mathcal{A}_1,\dots,\mathcal{A}_k$ have high probability, so does $\mathcal{D}_k$.

  From Corollary \ref{cor:separation_of_discrepancies}, we see that the event
  
  \[\mathcal{A}_1=\{\text{There is }\ov{x}\text{ such that }\abs{\ov{x}}<  n^{\zeta} \text{ and }\discr{\ov{x}} < b n^{\zeta-1}\}\]
has high probability.

  On $\mathcal{A}_1$, we consider the path that moves ballistically up to $\ov{x}$ and then remains optimally on the edge $\{\ov{x}, \ov{x}+1\}$.  By \eqref{eq:eta_lower_bound}, it has action at most
  \begin{equation}
  \label{eq:ballistic-action-at-most}
  c  n^\zeta+(n -  n^{\zeta}) ( b n^{\zeta-1}-c)\le -cn+(2c+b)n^\zeta.    
  \end{equation}

  Thus, the action of the minimal path $\maxPath_n$ does not exceed this quantity. In other words, $A(B,F,\newbar{\gamma}_n)+cn< k_2 n^\zeta$ for all $k_2>2c+b$.
  
  Using this fact and Lemma \ref{lm:bound_path_by_min_discr}, we can give an upper bound on the smallest discrepancy $\maxPath_n$  must reach, that is,
 \begin{align*}
  A(B,F,\newbar{\gamma}_n) +cn & \ge   \sum_{\substack{i \in (0,n) \\ i \textrm{ even}}} \1 \left\{ B(i) \neq B(i+1) \right\} \cdot \min_{x \in \range(\maxPath_n)} \discr{x}.
  \end{align*}

We can bound the sum of indicators above using the law of large numbers, and obtain that the event
  \begin{equation*}%\label{eq:lower0}
    \mathcal{A}_2=\left\{A(B,F,\newbar{\gamma}_n)+cn\ge   
    %\left( 
    \frac{1}5 n \cdot
    %\right) 
    \min_{x \in \range(\maxPath_n)} \discr{x}\right\}
  \end{equation*}
  has high probability.
  
  On the high probability event $\mathcal{D}_2$, we have
  \[(2c+b)n^\zeta\ge A(B,F,\newbar{\gamma}_n)+cn\ge   \frac{1}5 n d_n.\]
  Rearranging, we get $d_n\le 5(2c+b) n^{\zeta-1}.$ Thus, the event
  \begin{equation*}
  \mathcal{A}_3=\{d_n\le k_2 n^{\zeta-1}\}
\end{equation*}
 has high probability.

  From Theorem \ref{thm:convtoPPP}, the event
  \[
  \mathcal{A}_4=\left\{ \mu_n \left( [-k_1 n^{\zeta}, k_1 n^{\zeta} ] \times [0, 5(2c+b)]  \right) = 0 \right\}
  \]
   has high probability.  
  We have seen that on $\mathcal{D}_2$, we have $d(\ell_n)=d_n\le 5(2c+b)$. Thus, if $\mathcal{A}_4$ also occurs, we  must have $|\ell_n|\ge k_1 n^\zeta$. Thus,
  \[\mathcal{A}_5=\{|\ell_n|\ge k_1 n^\zeta\}\]
  has high probability. 
  
  Finally, we give an upper bound for the range of $\maxPath_n$.  
  Any path with range not contained in $(-k_2 n^{\zeta}, k_2 n^{\zeta})$ moves through all sites in either $[1, k_2 n^{\zeta}]$ or $[ -k_2 n^{\zeta}, -1]$.  Thus, we can bound the action of any such path from below by 
  \begin{equation*}
    \min \left( \sum_{i = 1}^{k_2 n^{\zeta}} -\abs{\spacevar{i}}, \sum_{i = -1}^{-k_2 n^{\zeta}} -\abs{\spacevar{i}} \right) - (n - k_2 n^{\zeta}) c.
  \end{equation*}
  The law of large numbers gives that the event 
  \[
 \mathcal{A}_6=\left\{  \min \left( \sum_{i = 1}^{k_2 n^{\zeta}} -\abs{\spacevar{i}}, \sum_{i = -1}^{-k_2 n^{\zeta}} -\abs{\spacevar{i}} \right) - (n - k_2 n^{\zeta}) c\ge -cn + \left( \frac{c - \Exp \abs{ \spacevar{0}}}{2} \right) k_2 n^{\zeta}\right\}
  \]
  has high probability.
  
  Assuming that $k_2$ is large enough, we have 
  \[
    k_2 \frac{c - \Exp \abs{ \spacevar{0}}}{2} > 2c + b.
  \]
  Therefore, for those values of $k_2$, on $\mathcal{D}_6$, a path with range not contained in $(-k_2 n^{\zeta}, k_2 n^{\zeta})$ has greater action than $-cn+(2c + b)n^\zeta$, which is bounded below by the action of the ballistic path to $\ov{x}$, 
  see~\eqref{eq:ballistic-action-at-most},  and thus cannot have minimal action. In particular, \[\mathcal{A}_7=\{|\ell_n|\le k_2 n^{\zeta}\}\]
  has high probability. On $\mathcal{A}_7$, we have $d_n\ge\min_{x\in [-k_2 n^{\zeta}, k_2 n^{\zeta}]}d(x)$. Combining this with Theorem~\ref{thm:convtoPPP}, we obtain
  that 
  \[\mathcal{A}_8=\Big\{d_n\ge\min_{x\in [-k_2 n^{\zeta}, k_2 n^{\zeta}]}d(x)\ge k_1 n^{\zeta-1}\Big\}\]
has high probability. On $\mathcal{A}_2\cap\mathcal{A}_8$, we have 
\[A(B,F,\gamma)+cn\ge \frac{1}5 n d_n\ge \frac{1}5 k_1 n^{\zeta}.\]
Thus, the event
\[\mathcal{A}_9=\{A(B,F,\gamma)+cn\ge   k_1 n^{\zeta}\}\]
has high probability.

So far we proved \eqref{eq:free1}, \eqref{eq:free2}, \eqref{eq:free3}. To complete the proof of  Theorem \ref{thm:free}, it remains to prove that with probability tending to $1$, we have $\newbar{\gamma}_n(i)\in\{\ell_n,\ell_n+1\} $ for all $i\ge \tau_n$.

  In the rest of this section, we will consider events $\mathcal{A}$ which, besides the already mentioned 
  parameters $b,k_2,k_1$ and $n$, also may depend on positive real parameters $\delta$ and $\lambda$. 
  We say that~$\mathcal{A}$  has high probability if for every $\varepsilon>0$, we can choose a positive number $b^*$ and functions 
  \begin{align*}
  k_2^*&:(0,+\infty)\to(0,+\infty), \\ 
  k_1^*&:(0,+\infty)^2\to(0,+\infty),\\
  \delta^*&:(0,+\infty)^3\to(0,+\infty),\\ 
  \lambda^*&:(0,+\infty)^4\to (0,+\infty),
  \end{align*}
  such that for all $b\in(b^*,+\infty)$, all $k_2\in(k_2^*(b),+\infty)$, all $k_1\in(0,k_1^*(b,k_2))$, all $\delta\in(0,\delta^*(b,k_2,k_1))$, and all $\lambda\in(0,\lambda^*(b,k_2,k_1,\delta))$, we have $\Prob{\mathcal{A}}>1-\varepsilon$ for all large enough~$n$. 
On the event $\mathcal{D}_9$, we have $d_n<k_2 n^{\zeta-1}$  and $\range \left( \maxPath_n \right) \subset [-k_2 n^{\zeta}, k_2 n^{\zeta}]$. Combining this with Corollary \ref{cor:separation_of_discrepancies}, we see that the event % also assume that for all $\ell_n\neq x \in \range \left( \maxPath_n \right)$,
\[\mathcal{A}_{10}=\{d(x) - d_n > \delta n^{\zeta-1}\text{  for all $x \in \range \left( \maxPath_n \right)\setminus \{\ell_n\}$}\}\]
has high probability.
  
  Now, consider the set %of pairs with a given small discrepancy
  \[
      S = \bigcup_{\substack{x\in [-k_2n^{\zeta}, k_2n^{\zeta})\\\discr{x} \le 5k_2 n^{\zeta-1}}}\{x, x+1\} .
  \]

Let \[
    \mathcal{A}_{11}=\Big\{|F(x-1)|<c-\lambda\text{ and }|F(x+2)|<c-\lambda\text{ for all $x \in [-k_2n^{\zeta}, k_2n^{\zeta}]$ such that $d(x)\le 5k_2 n^{\zeta-1}$}\Big\}.
    \]

  \begin{lemma}
        The event $\mathcal{A}_{11}$ has high probability.
  \end{lemma}

  \begin{proof}
     Due to Lemma~\ref{lem:discrep_density_near_0} and \eqref{zetaidentity}, there is $c_2$ such that $\Prob{d(x)\le 5k_2 n^{\zeta-1}}\le c_2 n^{-\zeta}$. 
     Moreover, $d(x)$ is independent from $F(x-1)$ and $F(x+2)$. Thus,
     \begin{align*}\ProbBig{|F(x-1)|\ge c-\lambda\text{ and }d(x)\le 5k_2 n^{\zeta-1}}&\le c_2 n^{-\zeta}\Prob{|F(x-1)|\ge c-\lambda},\\
     \ProbBig{|F(x+2)|\ge c-\lambda\text{ and }d(x)\le 5k_2 n^{\zeta-1}}&\le c_2 n^{-\zeta}\Prob{|F(x+2)|\ge c-\lambda}.
     \end{align*}
     Due to the union bound,
     \[\Prob{\mathcal{A}_{11}} \ge 1-2(2k_2 n^{\zeta}+1)c_2 n^{-\zeta}\Prob{|F(0)|\ge c-\lambda}.\]
     Since $\lim_{\lambda\to 0} \Prob{|F(0)|\ge c-\lambda}=0 $, the statement follows.
  \end{proof}

    Let $\partial S$ be the neighbors of the set $S$, that is,
  \[
    \partial S = \{x \in [-k_2n^{\zeta}, k_2n^{\zeta}] \setminus S \,:\, \{x-1, x+1\} \cap S \neq \emptyset\}.
  \]
  On $\mathcal{A}_{11}$, $S$ consists only of disjoint pairs and for all $x \in \partial S$, we have 
  \[
      \abs{\spacevar{x}} < c - \lambda.
  \]
The next lemma implies that it is unlikely for $\maxPath_n$ to make short excursions away from $S$.

Recall that $\eta_{x}$ was defined in \eqref{eq:eta} as the optimal path restricted to the edge $\{x,x+1\}$.
\begin{lemma}\label{long ex}
Let $[s_1, s_2]\subset[1,n]$ be an interval such that $\maxPath_n(s_1) \in \partial S$, $\maxPath_n\left( (s_1, s_2] \right) \cap S = \emptyset$ and
\[\sum_{i=s_1}^{s_2} B(i)F(\eta_{\ell_n}(i)) > \sum_{i=s_1}^{s_2} B(i)F(\maxPath_n(i)).\]
On the event $\mathcal{A}_{11}$, we have
\[
    s_2 - s_1 +1> \lambda \discr{{\ell_n}}^{-1} > \lambda k_2^{-1} n^{1-\zeta}.
\]
\end{lemma}
\begin{proof}Due to \eqref{eq:eta_lower_bound}, we have 
\begin{align*}
  (\discr{{\ell_n}}-c)(s_2 - s_1+1)  \geq \sum_{i=s_1}^{s_2} B(i)F(\eta_{\ell_n}(i)) > \sum_{i=s_1}^{s_2} B(i)F(\maxPath_n(i))  >  -(c-\lambda) - c(s_2 - s_1).
\end{align*}
Rearranging, we obtain the desired statement.
\end{proof}

It follows from the Hoeffding inequality that the events
\begin{multline*}\mathcal{A}_{12}=\bigg\{\text{For all intervals $[s_1,s_2]\subset [1,n]$ such that $s_2 - s_1+1  > \lambda k_2^{-1} n^{1-\zeta}$, we have}\\\sum_{\substack{i \in [s_1,s_2) \\ i \textrm{ even}}} \1 \left\{ B(i) \neq B(i+1) \right\}\ge \frac{s_2 - s_1+1}{5}\bigg\}\end{multline*}
and
\begin{multline*}\mathcal{A}_{13}=\bigg\{\text{For all intervals $[s_1,s_2]\subset [1,n]$ such that $s_2 - s_1  +1> \lambda k_2^{-1} n^{1-\zeta}$, we have}\\\left|N^+(s_1,s_2)-\frac{s_2-s_1}{2}\right|< \frac{\delta(s_2 - s_1)}{12k_2}\bigg\}\end{multline*}
have high probability.

We claim that on the event $\mathcal{D}_{13}$, we have $\maxPath_n(i)\in\{\ell_n,\ell_n+1\}$ for all $i\ge \tau_n$. Let us prove this claim by contradiction. Assume that $\maxPath$ leaves $\{\ell_n,\ell_n+1\}$ for some time interval $(t_1,t_2]$. More precisely, let us assume that $\maxPath_n(t_1)\in \{\ell_n,\ell_n+1\}$, $\maxPath_n((t_1,t_2])\cap \{\ell_n,\ell_n+1\}=\emptyset$ and either $t_2=n$ or  $\maxPath_n(t_2+1)\in \{\ell_n,\ell_n+1\}$.

Now, we decompose the path on $(t_1, t_2]$ into two types of intervals.  First, we can take intervals $[s_1, s_2]$ such that $s_1$ is the step before $\maxPath_n$ moves into $S$, and $s_2$ is either $t_2$ or the last step before the path moves back into 
  \[
    S^{\prime} = [-k_2n^{\zeta}, k_2n^{\zeta}] \setminus S.
  \]
 Precisely, $\maxPath_n(s_1)$ and $\maxPath_n(s_2+1)$ are in $S^{\prime}$ and $\maxPath_n\left( [s_1+1, s_2] \right) \subset S$ (Note that then $\maxPath_n(s_1) \in \partial S$).  Denote the set of such intervals of this type on $(t_1, t_2]$ by $\mathcal{I}$. 

The intervals which remain after considering intervals of this type are intervals of steps for which $\maxPath_n$ is contained entirely in $S^{\prime}$, whose leftmost step is an entry of $\maxPath_n$ into $S^{\prime}$ from $S$ (hence is a step in $\partial S$).  The last point in the interval will be two steps before $\maxPath_n$ moves into $S$, by construction of the intervals in $\mathcal{I}$ (or $t_2$).  Call the set of these intervals $\mathcal{J}$.  The intervals in $\mathcal{I}$ and $\mathcal{J}$ form a partition of the points in $(t_1, t_2]$.  

Let us consider the path coinciding with $\maxPath_n$ outside of the excursion time interval $(t_1,t_2]$ and coinciding with the optimal path $\eta_{\ell_n}$ staying in $\{\ell_n,\ell_n+1\}$ on  $(t_1,t_2]$. Comparing the actions of these two paths and using the optimality of $\maxPath_n$, we see that
\[\sum_{i=t_1+1}^{t_2} B(i)F(\eta_{\ell_n}(i)) > \sum_{i=t_1+1}^{t_2} B(i)F(\maxPath_n(i)).\]
Therefore, there is an interval $[s_1,s_2]\in \mathcal{I}\cup \mathcal{J}$ such that 
\begin{equation}\label{s1s2assumption}
 \sum_{i=s_1}^{s_2} B(i)F(\eta_{\ell_n}(i)) > \sum_{i=s_1}^{s_2} B(i)F(\maxPath_n(i)).   
\end{equation}

By Lemma~\ref{long ex}, $s_2-s_1+1>\lambda k_2^{-1} n^{1-\zeta}$. Thus, on the high probability event $\mathcal{A}_{12}\cap\mathcal{A}_{13}$, we~have
\begin{equation}\label{s1s2}
\sum_{\substack{i \in [s_1,s_2) \\ i \textrm{ even}}} \1 \left\{ B(i) \neq B(i+1) \right\}\ge \frac{s_2 - s_1+1}{5}\text{ and }\left|N^+(s_1,s_2)-\frac{s_2-s_1}{2}\right|< \frac{\delta(s_2 - s_1)}{12k_2}.
\end{equation}

First we consider the case that the interval $[s_1, s_2]$ in $\mathcal{J}$.  We use the simple bound from Lemma \ref{lm:bound_path_by_min_discr} to say that the action 
of $\maxPath_n$
on this interval satisfies
  \begin{align*}
   \sum_{i=s_1}^{s_2} B(i)F(\maxPath_n(i))&\ge  -c(s_2 - s_1 +1)  + \sum_{\substack{i \in [s_2,s_1) \\ i \textrm{ even}}} \1 \left\{ B(i) \neq B(i+1) \right\} \cdot \min_{x \in S^{\prime}} \discr{x} \\
     &>  -c(s_2 - s_1+1) +  \frac{(s_2 - s_1+1)}5   \left(5k_2 n^{\zeta-1}\right) \\ 
     &>  (s_2 - s_1+1)( \discr{{\ell_n}}-c)\ge \sum_{i=s_1}^{s_2} B(i)F(\eta_{\ell_n}(i)),
  \end{align*}
where we used~\eqref{eq:eta_lower_bound} in the last inequality. This contradicts \eqref{s1s2assumption}.

Now we consider the case where the interval $[s_1, s_2]$ is in $\mathcal{I}$.  To bound the path in such an interval, we only need to say that for long enough time periods ($s_2 - s_1 +1> \lambda k_2^{-1} n^{1-\zeta}$), the action of a path restricted to a single edge is primarily given by the length of the interval times half the discrepancy of the edge minus $c$. 

By construction, if  $\maxPath_n((s_1,s_2])\subset \{x, x+1\} \subset S$, then the action of $\maxPath_n$ on $(s_1, s_2]$ is at least the action of $\eta_x$ (the optimal path on the edge $\{x, x+1\}$) on $(s_1, s_2]$. We compare this action to the action of~$\eta_{\ell_n}$ using  \eqref{eq:eta_action} and the fact that on the high probability event $\mathcal{A}_{10}$, $\discr{x} - \discr{{\ell_n}} > \delta n^{\zeta-1}$ for all $x\in \range(\maxPath_n)\setminus{\ell_n}$. 

Note that 
\[B(s_1)F(\eta_{\ell_n}(s_1))\le -c+k_2 n^{\zeta-1}\le -c+5k_2 n^{\zeta-1}\le B(s_1)F(\maxPath_n(s_1)),\]
where we used the fact that $[s_1,s_2]$ is in $\mathcal{I}$, so $s_1\in S^{\prime}$.
Thus, we have
\begin{align*}
  \sum_{i=s_1}^{s_2} B(i)F(\eta_{\ell_n}(i)) &- \sum_{i=s_1}^{s_2} B(i)F(\maxPath_n(i))\\
  &<\sum_{i=s_1+1}^{s_2} B(i)F(\eta_{\ell_n}(i)) - \sum_{i=s_1+1}^{s_2} B(i)F(\eta_x(i)) \\
  &= \frac{1}{2}(s_2 - s_1)(\discr{{\ell_n}} - \discr{x}) \\&\qquad\qquad+ \left(  N^+(s_1,s_2) - \frac{s_2-s_1}{2} \right)\left( \spacevar{x} + \spacevar{x+1} - \left( \spacevar{{\ell_n}} + \spacevar{{\ell_n}+1} \right) \right) \\
  & < \frac{1}{2}(s_2 - s_1)(\discr{{\ell_n}} - \discr{x}) + \left|  N^+(s_1,s_2) - \frac{s_2-s_1}{2} \right|\left( \discr{x} + \discr{{\ell_n}} \right)  \\
  & < \left|  N^+(s_1,s_2) - \frac{s_2-s_1}{2} \right|6k_2 n^{\zeta-1}  - \frac{1}{2}(s_2 - s_1)(\delta n^{\zeta-1})<0.
\end{align*}
where at the last inequality, we used \eqref{s1s2}. 
 
This again gives a contradiction and completes the proof of our claim that the path $\maxPath$ stays in $\{\ell_n,\ell_n+1\}$ after $\tau_n$. 

\epf

\section{Distributional convergence -- The proof of Theorem~\ref{thm:weakconv}}
\label{sec:distributional_convergence}
Theorem~\ref{thm:free} implies that $|\ell_n|\to \infty$ in probability. We need a somewhat stronger statement.
\begin{lemma}\label{lntoinfty}
Almost surely, $|\ell_n|\to \infty$ and $\tau_n\to\infty$.
\end{lemma}
\begin{proof}
Assume that with positive probability there is $B$ such that $|\ell_n|<B$ infinitely often. Conditioned on this event, there are $x_l<0$ and $x_u>0$ such that $d(x_l),d(x_u)<\min_{|x|<B} d(x)$. Then if $|\ell_n|<B$, then $\newbar{\gamma}_n$ must stay in $(x_l,x_u)$. Thus, $A(B,F,\newbar{\gamma}_n)\ge -n\max_{x\in (x_l,x_u)} |F(x)|$, which contradicts the fact that $\lim_{n\to\infty } n^{-1}A(B,F,\newbar{\gamma}_n)=-c$ almost surely, since $\shape(0)=c$. Thus, the first claim of the lemma holds. The second one easily follows.
\end{proof}

Using Lemma~\ref{lemma3} and the fact that  $|\ell_n|\to \infty$, we see that
\begin{align*}
\newbar{A}(B,F;M |\ell_n|,\ell_n)&=-M \shape\left(\frac{1}M\right)|\ell_n|+o(\ell_n)\\
\newbar{A}(B,F;M |\ell_n|,\ell_n+1)&=-M \shape\left(\frac{1}M\right)|\ell_n|+o(\ell_n)
\end{align*}
Here $o(\ell_n)$ means a random variable $Z_n$ such that $Z_n/\ell_n\to 0$ almost surely.

For each $n$, consider the event that $\tau_n<M|\ell_n|$ and $\gamma_n(i)\in \{\ell_n,\ell_n+1\}$ for all $i\ge \tau_n$. The probabilities of these events tend to $1$ by our assumption \eqref{eq:conjecture-time-to-min-discr} and Theorem~\ref{thm:free}. On this event, we have $\newbar{\gamma}_n(M|\ell_n|)\in \{\ell_n,\ell_n+1\}$. Thus, on this event, we have
\begin{equation}\label{onevent}\newbar{A}(B,F,M |\ell_n|,\newbar{\gamma}_n(M|\ell_n|))=-M \shape\left(\frac{1}M\right)|\ell_n|+o(\ell_n).
\end{equation}
By the law of large numbers, $N^+(0,n)=\frac{n}{2}+o(n)$. 
This, along with Lemma~\ref{lntoinfty}, implies \[N^+(\tau_n,n)=N^+(0,n)-N^+(0,\tau_n)=\frac{n}2+o(n)-\left(\frac{\tau_n}2-o(\tau_n)\right)=\frac{n-\tau_n}{2}+o(n).\] 
Therefore, using \eqref{eq:eta_action2}, we see that 
\begin{align*}\sum_{i=M |\ell_n|+1}^n B(i)F(\eta_{\ell_n}(i))&=(n-M|\ell_n|)\left(\frac{d_n}2-c\right)+d_n o(n)\\&=(n-M|\ell_n|)\left(\frac{d_n}2-c\right)+n^{1-\zeta}d_n o(n^{\zeta}),\end{align*}
where we used the notation $d_n=d(\ell_n)$ introduced in~\eqref{eq:def_ell_d_tau}.

On the event above,
\begin{align}\label{Ag}A(B,F;\newbar{\gamma}_n)+cn&=\newbar{A}(B,F,M |\ell_n|,\newbar{\gamma}_n(M|\ell_n|))+\sum_{i=M |\ell_n|+1}^n B(i)F(\eta_{\ell_n}(i))+cn\\&=M\left(c- \shape\left(\frac{1}M\right)\right) \ell_n+n\frac{d_n}2+o(n^\zeta)(n^{-\zeta}|\ell_n|+n^{1-\zeta}d_n).\nonumber
\end{align}

By Theorem \ref{thm:free}, $n^{-\zeta}|\ell_n|+n^{1-\zeta}d_n$ is tight. Thus,
\begin{equation}\label{Agto0}
    \frac{A(B,F;\newbar{\gamma}_n)+cn}{n^\zeta}-g(n^{-\zeta}\ell_n,n^{1-\zeta}d_n)
\end{equation}
converges to $0$ in law, where  
\begin{equation}
\label{eq:g_and_s}
g(x,y)=s|x|+\frac{y}2\quad\text{ and }\quad s=M\left(c- \shape\left(\frac{1}M\right)\right).
\end{equation}

The next lemma gives us another expression for $s$. Note that this lemma also gives us Theorem~\ref{linearunderassumption}.%implies that if assumption of \eqref{eq:conjecture-time-to-min-discr} hold, then Theorem~\ref{shape function linear} can be extended to the case $\kappa\le 0$. 
\begin{lemma}
Under the assumption of \eqref{eq:conjecture-time-to-min-discr}, the shape function $\shape$ is linear on the interval~$\left[0,\frac{1}M\right]$. So, the definition
\eqref{eq:g_and_s} does not depend on the choice of $M$ in~\eqref{eq:conjecture-time-to-min-discr}, and
\[s=-\shape'(0+)=M'\left(c- \shape\left(\frac{1}{M'}\right)\right),\quad M'>M.\]
\end{lemma}
\begin{proof}
Let $M'>M$. For all $n$, on the event \[\mathcal{A}_n=\Big\{ \tau_n<M|\ell_n|\ \text{\rm and}\ \gamma_n(i)\in \{\ell_n,\ell_n+1\}\ \text{\rm for all}\ i\ge \tau_n\Big\},\]
we have
\[\frac{1}{|\ell_n|}\sum_{i=1}^{M'|\ell_n|}B(i)F(\newbar{\gamma}_n(i))\ge \frac{1}{|\ell_n|}\sum_{i=1}^{M|\ell_n|}B(i)F(\newbar{\gamma}_n(i))-(M'-M)c,\]
i.e.,
\[\frac{1}{|\ell_n|}\newbar{A}(B,F;M' |\ell_n|,\newbar{\gamma}_n(M'|\ell_n|))\ge \frac{1}{|\ell_n|}\newbar{A}(B,F;M |\ell_n|,\newbar{\gamma}_n(M|\ell_n|))-(M'-M)c.\]
Similarly to \eqref{onevent}, we obtain that on $\mathcal{A}_n$, 
\begin{equation}\label{linfromas}-M'\shape\left(\frac{1}{M'}\right)+o(1)\ge -M\shape\left(\frac{1}{M}\right)+o(1)-(M'-M)\shape(0). \end{equation}

Since $\Prob{\mathcal{A}_n}$ tends to $1$, we can choose a subsequence $n'$ such that $\lim \mathbbm{1}_{\mathcal{A}_{n'}}=1$ almost surely. Taking the limit of \eqref{linfromas} along $n'$, we obtain
\[-M'\shape\left(\frac{1}{M'}\right)\ge -M\shape\left(\frac{1}{M}\right)-(M'-M)\shape(0),\]
which together with the concavity of the shape function $\shape$ implies the statement.
\end{proof}

Let \[x_n=\argmin_{k\in \Z} g(n^{-\zeta}x,n^{1-\zeta}d(x)).\]

\begin{lemma}
We have $|x_n|\to\infty$ almost surely.
\end{lemma}
\begin{proof} We give an argument similar to the proof of Lemma~\ref{lntoinfty}.
Assume that with positive probability there is $B$ such that $|x_n|<B$ infinitely often. Then there is $\delta>0$ such that with positive probability $|x_n|<B$  and $d(x_n)\ge \min_{|x|<B} d(x)>\delta$ infinitely often. Then with positive probability $\min_{x\in \Z} g(n^{-\zeta} x,n^{1-\zeta}d(x))=g(n^{-\zeta} x_n,n^{1-\zeta}d(x_n))>\delta n^{1-\zeta}/2$ infinitely often. Thus, with positive probability, there are infinitely many values of $n$ such that for all $x\in [-\frac{n\delta}{4s},\frac{n\delta}{4s}]$, we have $d(x)>\delta/2$.  Indeed, for an $x\in [-\frac{n\delta}{4s},\frac{n\delta}{4s}]$ such that $d(x)\le \delta/2$, we would have
\[g(n^{-\zeta} x,n^{1-\zeta}d(x))=sn^{-\zeta} |x|+\frac{n^{1-\zeta}d(x)}2\le sn^{-\zeta} \frac{n\delta}{4s}+\frac{n^{1-\zeta}\delta/2}{2}\le \delta n^{1-\zeta}/2.\]
Therefore, with positive probability, we have $d(x)>\delta/2$ for all $x\in\mathbb{Z}$, which is a contradiction. 
\end{proof}

Theorem \ref{thm:free} implies that $g(n^{-\zeta}\ell_n,n^{1-\zeta}d_n)$ is tight, so \[n^{-\zeta}|x_n|\le \frac{\min_{x\in \Z}g(n^{-\zeta}x,n^{1-\zeta}d(x)) }s\le \frac{g(n^{-\zeta}\ell_n,n^{1-\zeta}d_n)}{s} \] must be tight. By a similar argument $n^{1-\zeta}d(x_n)$ is also tight.

Let $\gamma'=\argmin_{\gamma\in\Gamma(Mx_n,x_n)} A(B,F;\gamma)$, and let
\[\gamma''=\begin{cases}\gamma'(i)& i\le Mx_n\\
\eta_{x_n}(i)&Mx_n<i\le n.
\end{cases}\]

A similar calculation  as above shows that 
\begin{multline*} s|\ell_n|+n\frac{d_n}2+o(n^\zeta)(n^{-\zeta}|\ell_n|+n^{1-\zeta}d_n)=A(B,F;\newbar{\gamma}_n)+cn\\\le A(B,F;\gamma'')+cn=s |x_n|+n\frac{d(x_n)}2+o(n^\zeta)(n^{-\zeta}|x_n|+n^{1-\zeta}d(x_n))\end{multline*}
Indeed, the first equality was proved above in \eqref{Ag}. For the last equality, we bound the action of $\gamma'$ by combining Lemma~\ref{lemma3} and the fact that $|x_n|\to\infty$ almost surely. The action of the segment of $\gamma''$ after time $Mx_n$ is bounded by using \eqref{eq:eta_action2} and the law of large numbers.
Thus,
\begin{multline*}g(n^{-\zeta}x_n,n^{1-\zeta}d(x_n))\le g(n^{-\zeta}\ell_n,n^{1-\zeta}d_n)\\\le g(n^{-\zeta}x_n,n^{1-\zeta}d(x_n)) +o(n^{-\zeta}|\ell_n|+n^{1-\zeta}d_n+n^{-\zeta}|x_n|+n^{1-\zeta}d(x_n)).
\end{multline*}

By tightness of the argument of $o(\cdot)$, we see that $g(n^{-\zeta}
x_n,n^{1-\zeta}d(x_n))- g(n^{-\zeta}\ell_n,n^{1-\zeta}d_n)$ converges to $0$ in law.

Using Theorem \ref{thm:convtoPPP}, we see that $g(n^{-\zeta}x_n,n^{1-\zeta}d(x_n))$ converges to
\[\min_{(x,y)\in \supp(\nu)}g(x,y).\]

As we observed in \eqref{Agto0}, $\frac{A(B,F;\newbar{\gamma}_n)+cn}{n^\zeta}$ has the same limit.

The pushforward of $\nu$ by $g$ is also a Poisson point process with intensity 

\[r(x)=2\int_{0}^x \frac{p_{\kappa}q^22^{2\kappa+2}}{s} u^{2\kappa+1}du=\frac{p_{\kappa}q^22^{2\kappa+2}}{s(\kappa+1)}x^{2\kappa+2}.\]
Thus.
\[\Prob{\min_{(x,y)\in \supp(\nu)}g(x,y)\ge t}=\exp\left(-\int_0^t r(x)dx\right)=\exp\left(-\frac{p_{\kappa}q^22^{2\kappa+2}}{s(\kappa+1)(2\kappa+3)}t^{2\kappa+3}\right).\]
Let us define
\[\newtau=\left(\frac{p_{\kappa}q^22^{2\kappa+2}}{s(\kappa+1)(2\kappa+3)}\right)^{-\frac{1}{2\kappa+3}}.\]
Then 
\[\Prob{\frac{\min_{(x,y)\in \supp(\nu)}g(x,y)}{\newtau}\ge t}=\exp\left(-t^{2\kappa+3}\right),\]
which completes the proof of Theorem~\ref{thm:weakconv}. \epf

\section{Concluding remarks and open problems}\label{secopenproblems}

In this section we present heuristic arguments explaining the transition from the linear to the non-linear behavior of the shape function $\Lambda (\alpha)$. We also formulate a few open problems 
related to our directed polymer model.

We first address the Assumption 
(\ref{eq:conjecture-time-to-min-discr}) saying that asymptotically
the time $\tau_n$ of reaching the position $\ell_n$ by the optimal path
is bounded by $M\ell_n$, where $M$ is a positive constant. For a fixed 
realization of the spatial disorder $\{F(x), \, x \in \Z\}$,
one can think of the optimal path as a solution to the stochastic optimal control problem with respect to the temporal disorder given by the random sign $B(i), \, i \in \N$. The optimal path adjusts
its route from the origin to $\ell_n$ taking into account both $B(i)$ and
the discrepancies $d(x)$ of visited sites. When $d(x)$ is small, it
is beneficial to spend more time on an edge $\{x,x+1\}$ in order to achieve a "better" sequence of signs $B(j)$ for further movement.
Such waiting can be characterized by local time $t(x)$ for an edge
$\{x,x+1\}$. It is natural 
to expect that $t(x) \sim g(d(x))$ for some non-random function
$g(\epsilon)$ which diverges to $\infty$ as $\epsilon\to 0$.
Since the density for discrepancy near $0$ behaves as $\epsilon^{1+2\kappa}$, the 
above assumption is equivalent to the following estimate:
\begin{equation}
\label{integral converges}
\int_0^{2c}{g(\epsilon)\epsilon^{1+2\kappa}d\epsilon}<+\infty.
\end{equation}
Below we assume that $g(\epsilon)$ diverges as a power law as $\epsilon \to 0$,
namely $g(\epsilon) \sim G\epsilon^{-\gamma}$, where the exponent $\gamma$ is universal. In other words, it does not depend on the probability distribution of the spatial disorder $F$. In particular, we assume that $\gamma$ does not depend on $\kappa$. %we get $\gamma < 2 +2 \kappa$ as the condition for validity of 
%(\ref{integral converges}).  Combining this with %Theorem~\ref{shape function linear} stating that $\kappa>0$ %implies that there is a flat edge of the shape function near %zero, we obtain $\gamma<2$. 
With the above assumption on the asymptotic behavior of $g(\epsilon)$ the estimate (\ref {integral converges}) holds for $\kappa>\gamma/2 -1$. In this range of $\kappa$ the law of large numbers
would imply that $\tau_n \sim M_0\ell_n$, where $M_0$
is a non-random positive constant which depends on the probability distribution for $F$. 

Coming back to the behavior of the shape function, we consider a space interval $[0,[\alpha n]]$.  Denote  $\{\ell_n,\ell_n+1\}$ the edge with the
smallest value of the discrepancy. 
Namely, $d(\ell_n)<d(y)$ for all
$y \in [0,[\alpha n]]$. Provided that the time $n$ is large enough,
the optimal path will go in the optimal way from the origin to 
the edge $\{\ell_n,\ell_n+1\}$, will stay there for certain time, and then
go, again in the optimal way, from the edge $\{\ell_n,\ell_n+1\}$ to the
point $[\alpha n]$. The previous discussion imply that two pieces of the optimal path, the one until it reaches $\{\ell_n,\ell_n+1\}$ and the other one after the path leaves this edge, together will require time $M_0[\alpha n]$. Thus, if $n>M_0[\alpha n]$ the optimal path
will spend time $(1-M_0\alpha)n$ on the edge $\{\ell_n,\ell_n+1\}$. It is easy to see that in this 
case the shape function will have linear edge near the origin:
$$ \shape (\alpha) = C(1-M_0\alpha) + \shape(1/M_0)M_0\alpha,\quad \alpha\leq 1/M_0.  $$

On the other hand,  for $\alpha > 1/M_0$ the path does not have enough time to reach $\{\ell_n,\ell_n+1\}$ and to go from there to $[\alpha n]$ in the optimal way. Thus the path would need to compromise. This means that
the path will have to go ballistically through some edges. The process of
finding the optimal strategy in the case of lack of time is much more  subtle which will result in nonlinear behavior of the shape function. Since for $\kappa<\gamma/2-1$ the integral (\ref{integral converges}) diverges, we expect this type of behavior for all
$\alpha$ for those $\kappa$. In other words, we expect that 
the shape function does not have a linear piece for $-1<\kappa<\gamma/2-1$. Combining this with Theorem~\ref{shape function linear} stating that $\kappa>0$ implies existence of the flat edge of the shape function near the origin, we obtain $\gamma \leq 2$. 

Although Theorem~\ref{thm:weakconv} 
formally cannot be applied in the above range $-1<\kappa<\gamma/2-1$, we expect that a similar statement holds true for all $\kappa>-1$.
The only difference is that $s=M\left(c-\Lambda\left(\frac{1}{M}\right)\right)$ in the proof of Theorem~\ref{thm:weakconv}
should be replaced by 
$$s=\lim_{M\to \infty}{M\left(c-\Lambda\left(\frac{1}{M}\right)\right)}=-\frac{d\Lambda}{d\alpha}(0+).$$ Such limit exists due to
convexity of $\Lambda$. It also follows from \eqref{eq:rect_shape_bound1} of Theorem~\ref{shape function non-linear} that $s>0$.

\

We finish with a discussion of open problems related to our model.
Most of the questions discussed below are relevant also for other
models of a similar type. Namely when the spatial and temporal disorder are separated, and their interaction can be described
as $\Phi(F(x),B(i))$, where $\Phi$ is a deterministic function.

Above we considered statistics of optimal paths corresponding to
either fixed time $n$, or fixed final position $\ell_n$.
In the case when disorder is independent for all points $(x,i) \in \Z^2$ one can prove existence of minimizers corresponding to infinite time (\cite {HN, BCK}). An infinite path $\bar\gamma(i), i \geq 0$
is called a one-sided minimizer at the origin if $\bar\gamma(0)=0$ and for any other path $\tilde\gamma$ such that $\tilde\gamma(0)=0$ and
there exists $N$ such that $\bar\gamma(i)=\tilde\gamma(i)$ for all $i\geq N$, the difference in action $A(B,F;\tilde\gamma_n) - A(B,F;\bar\gamma_n)\geq 0$ for all $n\geq N$. It is natural to expect
that similar result on existence and uniqueness of one-sided minimizers holds also in our model. But currently it is an open problem.

Another set of question is related to the case of positive temperature. Then one has to study asymptotic properties as $n \to \infty$ of Gibbs measures $P_\beta(\gamma_n)$: 
$$P_\beta(\gamma_n) = \frac{1}{Z_n(\beta)}\exp{(-\beta A(B,F; \gamma_n)}, \, Z_n(\beta) = \sum _{\gamma_n} \exp{(-\beta A(B,F; \gamma_n)}.$$ 
In the case of completely independent disorder such questions were
studied by Carmona, Hu and Comets, Shiga, Yoshida (\cite{CH, CSY}). It was proved
that in the case of spatial dimensions one and two the polymer
measure $P_\beta$ exhibits strong localization properties.
On the other hand starting from dimension three there is a transition from diffusive behavior for small $\beta$ (weak disorder) to strong disorder for $\beta >\beta_{cr}$.
Again, one can expect that some of the above results can be extended
to the case of the model studied in this paper. In fact, we would expect that localization
properties are stronger in our model.

Let us finally mention the problem which is related to our numerical studies.
We already mentioned that statistical properties may be studied
separately for spatial and temporal disorder. Consider a fixed realization of spatial disorder $F$. Denote by $e_i$ a sequence
of edges $e_i=\{x_i,x_i+1\}, x_i\geq 0$ corresponding to records of
the discrepancy $d$. Namely, $d_i=d(x_i) < d(x)$ for all $0\le x<x_i$. Fix large $n$ and consider all paths $\gamma=\gamma (i), \, i \geq 0$ 
such that they go from the origin to the edge $e_n$ and then stay
on this edge jumping according to the sign of $B$.  We also assume that $0\leq \gamma(i) \leq x_n+1$. Notice that we do not fix
the time $m$ when the edge $e_n$ is reached for the first time.
Denote by $A_n(B)$ the minimal action accumulated on the way from the origin to $e_n$:
$$A_n(B)=\min_{\gamma}A_{m(\gamma)}(B,F;\gamma).$$
We studied numerically the variance of $A_n(B)$ for a fixed
realization of $F$. According to our numerics the variance does not
grow with $n$. It remained bounded even when $x_n$ was order $10^4$.
At this moment we don't have  good explanation for this phenomenon.

\begin{appendix}

\section{Detailed proof of Lemma \ref{th:shape_function_exists}} \label{sec:appendix}

	From the construction of the model, we have the following bound.  
	For points $(t_1, x_1), (t_2, x_2) \in \cone \cap \Z^2$, we have
	\begin{align*}
		| \rectAction \left( (t_1, x_1), (t_1 + t_2, x_1 + x_2) \right) | \leq ct_2,
	\end{align*}
	and so subadditivity (see \eqref{eqsuperadditivity}) gives the following: for $(t_1, x_1), (t_2, x_2) \in \cone \cap \Z^2$, we have 
	\begin{align}
		\label{eq:action_bound}
		\rectAction(t_1, y_1) \geq \rectAction(t_1 + t_2, x_1 + x_2) - ct_2.
	\end{align}

Following  \cite{seppalainen2009lecture}, we first prove that for $(t,x)\in \cone \cap \Z^2$, $n^{-1}A^*(nt,nx)$ converge almost surely to a constant $\rectShape(t,x)$, then extend this to rational points, and finally all real numbers.  
Thus, first suppose that $(t, x) \in \cone \cap \Z^2$.  If $x=0$,  Lemma~\ref{lem:shape-at-0} gives that \[\lim_{n\to\infty} n^{-1} A^*(nt,0)=-ct\text{ almost surely}.\]

To handle the case $x\neq 0$, we will use Kingman's subadditive  ergodic theorem \cite{kingman1968ergodic,kingman1973subadditive}. Below we state Liggett's version of this theorem~\cite{liggett1985improved}. 

% \begin{theorem}
% Suppose $( X_{m,n} )$ is a sequence of random variables indexed by integers $0\leq m < n$, such that 
% \begin{enumerate}
% 	\item 
		
%   \hfil
%    $X_{l,n} \leq X_{l,m}+X_{m,n} \text{ whenever } 0\leq l < m < n.
% 		$\hfil
% 	\item The joint distributions of $\left( X_{m+1, n+1}, 0 \leq m < n \right)$ are the same as those of $\left(X_{m, n}, 0 \leq m < n \right)$.
% 	\item For each $n$, $\Exp|X_{0,n}| < \infty$ and $\Exp X_{0,n} \geq -cn$ for some constant $c$.
% \end{enumerate}
%  Then the limit $X = \lim_{n \to \infty} \frac{1}{n} X_{0,n} $ exists almost surely and in $L^1$, and 
%  \[
%  	\Exp X = \lim_{n \to \infty} \frac{1}{n} \Exp X_{0,n} = \inf_n  \frac{1}{n} \Exp X_{0,n}
%  \]
%  Furthermore, if the sequence $(X_{0,n})$ is ergodic, $X = \Exp X$ a.s.  
% \end{theorem}

\begin{theorem}\label{Kingman}
Suppose $( X_{m,n} )$ is a sequence of random variables indexed by integers $0\leq m < n$, such that 
\begin{enumerate}[(i)]
	\item 
		
  \hfil
   $X_{0,n} \leq X_{0,m}+X_{m,n} \text{ whenever } 0 < m < n.
		$\hfil
	\item The joint distributions of $\left( X_{m+1, m+k+1}, k\ge 1 \right)$ are the same as those of $\left( X_{m, m+k}, k\ge 1 \right)$ for each $m\ge 0$.

\item \label{statitem}For each $k\ge 1$, $\left( X_{nk, (n+1)k}, n\ge 1 \right)$  is a stationary process.
 
	\item For each $n$, $\Exp|X_{0,n}| < \infty$ and $\Exp X_{0,n} \geq -cn$ for some constant $c$.
\end{enumerate}
 Then the limit $X = \lim_{n \to \infty} \frac{1}{n} X_{0,n} $ exists almost surely and in $L^1$, and 
 \[
 	\Exp X = \lim_{n \to \infty} \frac{1}{n} \Exp X_{0,n} = \inf_n  \frac{1}{n} \Exp X_{0,n}
 \]
 Furthermore, if the stationary processes in \eqref{statitem}  are ergodic, $X = \Exp X$ a.s.  
\end{theorem}

Define $X_{m,n} = \rectAction \left( (mt, mx), (nt, nx) \right)$ for $0 \leq m < n$.  
% \todo{fill this in.}  
Then, by Theorem~\ref{Kingman}, we see that 
\[
n^{-1}X_{0,n} = n^{-1} \rectAction (tn, xn)
\]
tends almost surely to a limit, $\rectShape(t,x)$. The conditions of Theorem~\ref{Kingman} are easy check, only the ergodicity condition requires some explanation, which is given in the next lemma:

\begin{lemma}
Let $k\ge 1$. Then for all $h\ge 1$ and $m>2ht+1$, the random vectors
\[\left( X_{nk, (n+1)k}, \quad 1\le n\le h \right)\quad\text{ and }\quad\left( X_{nk, (n+1)k}, \quad m\le n\le m+h-1 \right)\]
are independent.

In particular, the process $\left( X_{nk, (n+1)k}, n\ge 1 \right)$  is strong mixing, thus, ergodic. 
\end{lemma}
\begin{proof}
By symmetry, we may assume that $x>0$. 

On the one hand, 
$\left( X_{nk, (n+1)k},\quad 1\le n\le h \right)$ only depends on the random variables $B(i)$, where $ i\le (h+1)kt$ and $F(y)$, where $ y\le xk+hkt$. 

On the other hand, $\left( X_{nk, (n+1)k},\quad m\le n\le m+h-1 \right)$  only depends on the random variables $B(i)$ where $ i\ge mkt$ and $F(y)$ where $ y\ge mkx-hkt$. 

Thus, by the choice of $m$,  $(\left( X_{nk, (n+1)k}, 1\le n\le h \right))$ and $\left( X_{nk, (n+1)k}, m\le n\le m+h-1 \right)$ depend on disjoint sets of independent random variables, thus, they are independent.
\end{proof}

For any $(t,x)\in \cone\cap \mathbb{Z}^2$ and $k \in \N$, the homogeneity property~\eqref{eq:homogeneity} is immediate. %We will extend it to rational and then to real numbers.

Subadditivity of $\Psi$ in $\cone \cap \Z^2$ follows by taking limits in~\eqref{eqsuperadditivity}.
Given an arbitrary path from the origin to $(nt,nx) \in \cone \cap \Z^2$, the expectation of the action of the path is $0$, implying that the minimum of action over all such paths, $X_{0,n}$, has expectation smaller than or equal to $0$. Thus, $\rectShape(t,x) = \lim_{n \to \infty} n^{-1} \Exp X_{0,n} \le 0$.

Combining this with subadditivity gives a form of monotonicity for $(t_1, x_1), (t_2, x_2) \in \cone \cap \Z^2$ for $\rectShape$: 
\begin{align}
	\label{eq:shape_bound_superadditive}
	\rectShape(t_1, x_1) \geq \rectShape(t_1 + t_2, x_1 + x_2).  
\end{align}

Now, suppose $(t,x) \in \cone \cap \Q^2$.  We can take any choose $q \in \N$ such that $(qt, qx) \in \Z^2$, and define
\[
	\rectShape(t,x) = q^{-1} \rectShape(qt,qx).
\]
This definition is independent of the choice of $q$ by homogeneity for integers.  With this definition, we have homogeneity for rational $k > 0$, subadditivity for points in $\cone \cap \Q^2$, as well as the monotonicty property \eqref{eq:shape_bound_superadditive} for points in $\cone \cap \Q^2$.  

Next, we prove relation \eqref{eq:LLN_shape_function} for rational points $(t,x) \in \cone \cap \Q^2$. Without loss of generality, let us assume that $x\ge 0$. Given any $n \in \N$, write $n = mq + r$, $r \in \{0, \ldots, q-1 \}$.  %Call $R_x = \round{nx} - mqx = \round{rx}$ and $R_y = \round{ny} - mqy = \round{ry}$.  Then $\abs{R_y} \leq R_x$ for $(x,y) \in \cone \cap \Q^2$, and thus $(R_x, R_y) \in \cone \cap \Z^2$.  

An elementary argument gives that
\[(\round{nt} - mqt,\round{nx} - mqx)\in \cone\text{ and }(  (m+1)qt-\round{nt},  (m+1)qx-\round{nx})\in \cone.\]
Then by \eqref{eq:action_bound}, we have 
\[
	 \rectAction(mqt, mqx) + c(\round{nt} - mqt)\ge \rectAction( \round{nt}, \round{nx} )  \ge  \rectAction((m+1)qt, (m+1)qx) - c((m+1)qt-\round{nt}) 
\]
and thus, dividing by $n$ and taking $ n \to \infty$, we obtain  \eqref{eq:LLN_shape_function} for rational points.

Finally, as in \cite{seppalainen2009lecture}, for all $(t,x) \in \cone$, extend  
\begin{align}\label{extendrecshape}
	\rectShape(t,x) = \inf\big\{\rectShape(u,v)\,:\,(u,v) \in \cone \cap \Q^2, (t-u, x-v) \in \cone\big\}.
\end{align}
This agrees with the case of rational $(t,x)$ due to \eqref{eq:shape_bound_superadditive}.  
  
We already established subadditivity for rational points in $\cone$, and it straightforwardly follows from the definition~\eqref{extendrecshape} that it extends to all points of $\cone$. Since $\rectShape$ is at most $0$, we also obtain \eqref{eq:shape_bound_superadditive}  using subadditivity. 

For arbitrary $(t,x) \in \cone$, it is easy to see that we have homogeneity for rational $k > 0$. For general $k>0$, we choose arbitrary $k_1,k_2\in\Q$ satisfying $k_1<k<k_2$, use \eqref{eq:shape_bound_superadditive} to write
\[
	\rectShape(k_1 t, k_1 x ) \geq \rectShape(kt, kx) \geq \rectShape(k_2 t, k_2 x).
\]
Then we use homogeneity for $k_1$ and $k_2$ finally take limits as $k_1, k_2 \to k$.  

  Convexity follows from superadditivity and homogeneity. 

A finite convex function on an open set is continuous, so that we have continuity in $\cone^\circ$, the interior of $\cone$.

To compute the limit \eqref{eq:LLN_shape_function} for general $(t,x) \in \cone^\circ$, we use rational approximations.  Let $t_1,t_2,x_1,x_2$ be rationals such that
\begin{align*}
&0 < t_1 < t < t_2,
\\
&x_1 < x < x_2,
\\
&(t_1, x_1),(t_2, x_2) \in \cone,
\\
&2(x - x_1) \leq (t - t_1),
\\
&2(x_2 - x) \leq (t_2 - t).
\end{align*}
For large enough $n$, we have $(\round{nt} - \round{nt_1}, \round{nx} - \round{nx_1}) \in \cone,\, (\round{nt_2} - \round{nt}, \round{nx_2} - \round{nx})\in \cone$.  
Then we can use subadditivity to write
\begin{align*}
	\rectAction(\round{nx_1}, \round{ny_1}) + c(\round{nx} - \round{nx_1})\ge \rectAction(\round{nx}, \round{ny})  \geq \rectAction(\round{nx_2}, \round{ny_2}) - c(\round{nx_2} - \round{nx}).
\end{align*}

Dividing by $n$ and taking the limit, we obtain 
\begin{align*}
	\rectShape(t_1, x_1) + c(t-t_1)\geq \limsup_{n \to \infty} \rectAction(\round{nt}, \round{nx})\geq \liminf_{n \to \infty} \rectAction(\round{nt}, \round{nx})  \geq \rectShape(t_2, x_2) - c(t_2 - t).
\end{align*}

Continuity of $\rectShape$ on $\cone^\circ$ implies that we can take $(t_1, x_1)$ and $(t_2, x_2)$ to $(t, x)$ to obtain \eqref{eq:LLN_shape_function} for general $(t,x) \in \cone^\circ$.  

Finally consider a boundary point  $(t,\pm t)$ of $\cone$. Then $\Gamma(\round{nt},\pm\round{nt})$ contains a single path, with expected action $0$. Thus, it is easy to see that \eqref{eq:LLN_shape_function} holds in this case too. 
Therefore, we have \eqref{eq:LLN_shape_function} on all of $\cone$. \epf

\end{appendix}

%%%%%%%%%%%%%%%%%%%%%%%%%%%%%%%%%%%%%%%%%%%%%%
%% Support information, if any,             %%
%% should be provided in the                %%
%% Acknowledgements section.                %%
%%%%%%%%%%%%%%%%%%%%%%%%%%%%%%%%%%%%%%%%%%%%%%
\bigskip

\textbf{Acknowledgements}

 The authors are grateful to the anonymous referees for their useful comments.  Yuri Bakhtin is grateful to NSF for partial support through grants  DMS-1811444, DMS-2243505, DMS-1460595. Konstantin Khanin was partially supported by the NSERC Discovery Grant RGPIN-2018-04510. Andr\'as M\'esz\'aros was supported by the NSERC discovery grant of B\'alint Vir\'ag, the KKP 139502 project, the Dynasnet European Research Council Synergy project -- grant number ERC-2018-SYG 810115, and the NKKP-STARTING 150955 project.

%%  imsart-number.bst  will be used to                     %%
%%  MR numbers will be added by VTeX.                      %%
%%                                                         %%
%%  Use \cite{...} to cite references in text.             %%
%%                                                         %%
%%%%%%%%%%%%%%%%%%%%%%%%%%%%%%%%%%%%%%%%%%%%%%%%%%%%%%%%%%%%%

%% if your bibliography is in bibtex format, uncomment commands:
\bibliographystyle{plain}
 % Style BST file
\bibliography{references}       % Bibliography file (usually '*.bib')
\end{document}